\documentclass[12pt]{article}
\usepackage{amsmath, amsfonts, latexsym, amssymb, graphicx, mathtools,  wasysym,}

\usepackage{hyperref}
\hypersetup{colorlinks,citecolor=blue,linkcolor=blue}

\title{ Relatively hyperbolic groups with free abelian second cohomology}
\author{Michael Mihalik and Eric Swenson}
\newtheorem{theorem}{Theorem}[section]

\newtheorem{lemma}[theorem]{Lemma}
\newtheorem{corollary}[theorem]{Corollary}
\newtheorem{remark}[theorem]{Remark}

\newcounter{claimnum}

\newcounter{definitionnum}
\newenvironment{definition}{\addvspace{12pt}\refstepcounter{definitionnum}
\noindent{\bf Definition \arabic{definitionnum}.}}{\par\addvspace{12pt}}

\newenvironment{proof}{\addvspace{12pt}\noindent{\bf Proof:}}{
$\Box$\par\addvspace{12pt}}

\date{\today}
\begin{document}

\maketitle

\begin{abstract} 
Suppose $G$ is a 1-ended finitely presented group that is hyperbolic relative to $\mathcal  P$ a finite collection of 1-ended finitely presented proper subgroups of $G$. Our main theorem states that if the boundary $\partial (G,{\mathcal P})$ is locally connected and the second cohomology group $H^2(P,\mathbb ZP)$ is free abelian for each $P\in \mathcal P$, then $H^2(G,\mathbb ZG)$ is free abelian. When $G$ is 1-ended it is conjectured that $\partial (G,\mathcal P)$ is always locally connected. Under mild conditions on $G$ and the members of $\mathcal  P$ the 1-ended and local connectivity hypotheses can be eliminated and  the same conclusion is obtained. When $G$ and each member of $\mathcal P$ is 1-ended and $\partial (G,\mathcal P)$ is locally connected, we prove that a ``Cusped Space" for this pair has semistable fundamental group at $\infty$. This provides a starting point in our proof of the main theorem. 
\end{abstract}

\section{Introduction}
We are interested in an old conjecture (probably  due to H. Hopf) speculating that $H^2(G,\mathbb ZG)$ is free abelian for any finitely presented group $G$. T. Farrell \cite{FFT} proved that if $G$ is finitely presented and contains an element of infinite order then $H^2(G,\mathbb ZG)$ is either $0$, $\mathbb Z$ or not finitely generated. A result of B. Bowditch \cite{Bo04} implies $H^2(G,\mathbb ZG)=\mathbb Z$ if and only if $G$ contains a nontrivial closed surface group as a subgroup of finite index. 
For $R$ a ring, the $R$-module $H^2(G,RG)$ is torsion free (see \cite{G} 13.7.1). In Section 13.8 of \cite{G}, an example of a finite aspherical 3-pseudomanifold is constructed. Hence, its fundamental group $G$ has type $F$ and geometric dimension 3, but $H^3(G,\mathbb ZG)$ is isomorphic to $\mathbb Z_2$. This was the first exhibited example of a group of type $F_n$ for which $H^n(G,\mathbb ZG)$ is not free abelian. 
The connection between $H^k(G,\mathbb ZG)$ and $H_{k-1}(\varepsilon G)$ (the $(k-1)$-homology of the end of $G$) was explored in \cite{GM85}, also see  Section 13.7 of Geoghegan's book \cite{G}.  Our main theorem is:

\begin{theorem} \label{Hom} 
Suppose a finitely presented 1-ended group $G$ is hyperbolic relative to $\mathcal P=\{P_1,\ldots ,P_n\}$ a set of 1-ended finitely presented  subgroups (with $G\ne P_i$ for all $i$). If  the boundary $\partial (G, \mathcal P)$ is locally connected and for each $i$, $H^2(P_i,\mathbb ZP_i)$ is free abelian, then  $H^2(G;\mathbb ZG)$ is free abelian.
\end{theorem} 

When $G$ is 1-ended and hyperbolic relative to $\mathcal P$ then it may always be the case that $\partial (G,\mathcal P)$ is locally connected.  Note that there is no hypothesis on the number of ends of the $P_i$ and no local connectedness hypotheses on $\partial (G,\mathcal P)$ in the next result.

\begin{corollary} \label{Cor} 
Suppose a 1-ended finitely presented group $G$ is hyperbolic relative to $\mathcal P=\{P_1,\ldots ,P_n\}$ a set of finitely presented subgroups (with $G\ne P_i$ for all $i$). If  for each $i$, $P_i$ contains no infinite torsion subgroup and $H^2(P_i,\mathbb ZP_i)$ is free abelian, then $H^2(G;\mathbb ZG)$ is free abelian.
\end{corollary}

Section \ref{SS} explores what it means for a space and group to have semistable fundamental group/first homology at $\infty$. If $G$ is a 1-ended finitely presented group and $X$ is some (equivalently any) finite CW complex with $\pi_1(X)=G$ and universal cover $\tilde X$, then $G$ has semistable fundamental at $\infty$ if for any two proper rays $r,s:([0,\infty),\{0\})\to (\tilde X,\ast)$, there is a proper homotopy $H:[0,\infty) \times [0,\infty)\to \tilde X$ with $H(t,0)=r(t)$ and $H(0,t)=s(t)$ for all $t\in [0,\infty)$. We say $G$ has semistable first homology at $\infty$ if $r$ and $s$ are properly homologous. Certainly, if  a finitely presented group $G$ has semistable fundamental group at $\infty$ then $G$ also has semistable first homology at $\infty$. The group $H^2(G,\mathbb ZG)$ is free abelian if and only if $G$ has semistable first homology at $\infty$ (see Corollary \ref{Free}). At present, it is unknown if all finitely presented groups have semistable fundamental group at $\infty$, but there are many classes of groups (including the class of word hyperbolic groups) which are known to only contain groups that have semistable fundamental group at $\infty$. At this point we know of only one result other than our main theorem that concludes all members $G$ of a class of groups have free abelian second cohomology with $\mathbb ZG$ coefficients without passing through a fundamental group at $\infty$ result. I. Biswas and M. Mj \cite{BM17} prove that if $G$ is a holomorphically convex  group (in particular, if $G$ is a linear projective group) then $H^2(G,\mathbb ZG)$ is free abelian.  

Combining work of M. Bestvina and G. Mess \cite{BM91}, B. Bowditch \cite{Bow99B} and G. Swarup \cite{Swarup} one can conclude that if $G$ is a word hyperbolic group then $G$ has semistable fundamental group at $\infty$ (and so $H^2(G, \mathbb ZG)$ is free abelian).  Interestingly, our proof of Theorem \ref{Hom} does not translate into a proof that $G$ has semistable fundamental group at $\infty$ when the peripheral subgroups $P_i\in \mathcal P$ have semistable fundamental group at $\infty$. Basically the problem is that fundamental group is a pointed functor (while first homology is not). Still, the main theorem of \cite{MS18} is: 

\begin{theorem} \label{MS} (\cite{MS18}, Theorem 1.2) 
Suppose $G$ is a 1-ended finitely generated group that is hyperbolic relative to a collection of 1-ended finitely generated proper subgroups $\mathcal P=\{P_1,\ldots, P_n\}$. If $\partial (G,{\mathcal P})$ has no cut point, then $G$ has semistable fundamental group at $\infty$.
\end{theorem}

Note that there is no semistability hypotheses on the peripheral subgroups $P_i$ in this last result.  On the other hand, any time $G$ splits over a proper subgroup of a peripheral  subgroup, $\partial (G,\mathcal P)$  contains a cut point. Theorem \ref{Bow1}  implies that in most cases, when $G$ is 1-ended $\partial (G,\mathcal P)$ is locally connected {\it even} when it contains cut points. 
 
Let $X$ be the cusped space for  $(G, \mathcal P)$ (defined in \S \ref{GMSpace}), and $Y\subset X$ be the Cayley 2-complex for $G$ (with finite presentation containing a subpresentation  for each of the parbolics).  The pair $(G,\mathcal P)$ is relatively hyperbolic if and only if $X$ is $\delta$-hyperbolic. Our first goal is to prove the following theorem, which is an important component in our proof of Theorem \ref{Hom}. The semistability of the fundamental group at $\infty$ of the cusped  space for $(G, \mathcal P)$ does not seem to be of much help in showing that $G$ has semistable  fundamental group at  $\infty$. 

\begin{theorem} \label{GM-SS} 
Suppose a finitely presented group $G$ is hyperbolic relative to $\mathcal P=\{P_1,\ldots ,P_n\}$ a set of 1-ended finitely presented subgroups (with $G\ne P_i$ for all $i$). If  the boundary $\partial (G, \mathcal P)$ is locally connected then the cusped space for $(G,\mathcal P)$ has semistable fundamental group at $\infty$. 
\end{theorem}

When $G$ is word hyperbolic it acts simplicially on a contractible locally finite and finite dimensional $\delta$-hyperbolic simplicial complex $K$ usually called the Rips complex. The group $G$ acts freely and transitively on the vertices of $K$, and the quotient space $G/K$ is compact so that $\partial G=\partial K$. Combining \cite{Bow99B} and \cite{Swarup} we see that when $G$ is 1-ended then $\partial G$ is locally connected. Applying \cite{BM91} we see $\partial K$ is a $Z$-set in $K$ and so; the word hyperbolic group $G$ has semistable fundamental group at $\infty$ if and only if $\partial K(=\partial G)$ has the shape of a locally connected continuum. Unfortunately, the analogous approach for relatively hyperbolic groups comes up a bit short for our purposes. 
The set $\partial (G,\mathcal P)$ need not be a $Z$-set for a cusped space, but when the pair $(G,\mathcal P)$ is $F_\infty$,  Corollary 3.17 of \cite{MW18} states that for $n>0$, there is an $(n+1)$-dimensional cusped space $X_{n+1}$ such that $\partial (G,\mathcal P)$ is a ``$Z_n$-set" in $X_{n+1}\cup \partial (G,\mathcal P)$. In the case $n=1$ and when $\partial (G,\mathcal P)$ has the shape of a locally connected continuum, the $Z_n$-set conclusion implies the corresponding cusped space $X_2$ has semi-stable fundamental group at $\infty$. Manning and Wang remark that the $F_\infty$ hypothesis of their Corollary 3.17 can easily be relaxed to an $F_{n+2}$ hypothesis (on $G$ and each member of $\mathcal P$) and one still obtains that $\partial (G,\mathcal P)$ is a $Z_n$ set for $X_{n+1}$.  In particular, if $G$ and each member of $\mathcal P$ is $F_3$, then $X_2$ has semistable fundamental group at $\infty$. It seems this approach will not produce a semistability result in the more general setting of finitely presented groups. Hence the need for Theorem \ref{GM-SS} is apparent. 

In Section \ref{CorP} we list three established results from the literature that along with Theorem \ref{Hom} and Duwoody's accessibility result directly imply Corollary \ref{Cor}. 
Section \ref{SS}, is devoted to basic definitions and background for the semistability of the first homology and fundamental group at $\infty$ of a finitely presented group $G$, including the connection to $H^2(G,\mathbb ZG)$.  Section \ref{GMSpace} covers basic definitions and background for relatively hyperbolic groups and cusped spaces. Section \ref{T2} contains our proof of Theorem \ref{GM-SS}. In section \ref{SimpApp} we prove two simplicial approximation results for $[0,\infty)\times [0,1]$. The first is a technical result with applications beyond this article. Section \ref{T1} concludes with a proof of the Main Theorem \ref{Hom}. 

\medskip

\noindent{\bf Acknowledgements:} We are grateful to J. Manning for helpful conversations concerning the relaxation of $F_n$ hypothesis in his work (discussed near the end of this section). 

\section{The Proof of Corollary \ref{Cor}} \label{CorP}

Several results mesh well with our main theorem and combine to imply Corollary \ref{Cor}.

\begin{theorem} (\cite {Bow01}, Theorem 1.5) \label{Bow1} 
Suppose $(G, {\mathcal P})$ is relatively hyperbolic, $G$ is 1-ended and each $P\in {\mathcal P}$ is finitely presented, does not contain an infinite torsion group, and is either 1 or 2-ended, then $\partial (G,{\mathcal P})$ is locally connected.
 \end{theorem} 

\begin{theorem} \label{split1} 
(see \cite{DS05}, Corollary 1.14; or \cite{Osin06})
Suppose the group $G$ is finitely generated and hyperbolic relative to the finitely generated groups $P_1,\ldots, P_n$. If $\mathcal P_i$ is a finite graph of groups decomposition of $P_i$ such that each edge group of $\mathcal P_i$ is finite, then $G$ is also hyperbolic relative to the subgroups $\{P_1,\ldots, P_{i-1}, P_i,\ldots, P_n\}\cup V(\mathcal P_i)$ where $V(\mathcal P_i)$ is the set of vertex groups of $\mathcal P_i$. 
\end{theorem}

If a peripheral  subgroup $P_i$ is either finite or 2-ended, it may be removed from the collection of peripheral  subgroups and $G$ remains hyperbolic relative to the remaining subgroups. Recall that a finitely generated group is {\it accessible} if it has a finite graph of groups decomposition with each edge group finite and each vertex group either finite or 1-ended. By M. Dunwoody's accessibility theorem \cite{Dun85}, all (almost) finitely presented groups are accessible.  We show there is a finite collection $\mathcal P'$ of 1-ended subgroups of $G$ (each distinct from $G$) such that $\partial (G,\mathcal P')$  is locally connected and $H^2(P,\mathbb ZP)$ is free abelian for each $P\in \mathcal P'$. Then apply Theorem \ref{Hom}. 

For each $i$ let $\mathcal G_i$ be a Dunwoody decomposition of $P_i$. Let $\mathcal P_1$ be the set of vertex groups of the $\mathcal G_i$ for all $i\in \{1,\ldots, n\}$. By Theorem \ref{split1}, $G$ is hyperbolic relative to $\mathcal P_1$. By the following result, if $P\in \mathcal P_1$ is 1-ended, then $H^2(P;\mathbb ZP)$ is free abelian. 

\begin{theorem}\label{Dunw}  (\cite{M87}, Theorem 4) 
Suppose $G$ is a finitely presented group and $\mathcal G$ is a finite graph of groups decomposition of $G$ such that each vertex group is either finite or 1-ended and each edge group is finite. Then $H^2(G,\mathbb ZG)$ is free abelian if and only if  for each 1-ended vertex group $V$ of $\mathcal G$, $H^2(V,\mathbb ZV)$ is free abelian.
\end{theorem}

Let $\mathcal P'$ be obtained from $\mathcal P_1$ by removing all finite groups. Again, $G$ is hyperbolic relative to $\mathcal P'$.     By Theorem \ref{Bow1}, $\partial (G,\mathcal P'$) is locally connected, and by Theorem \ref{Hom}, $H^2(G,\mathbb ZG)$ is free abelian.

\section{Semistability at $\infty$} \label{SS}

The best reference for the notion of semistable fundamental group (homology) at $\infty$ is \cite{G} and we use this book as a general reference throughout this section. While semistability makes sense for multiple ended spaces, we are only interested in 1-ended spaces in this article. Suppose $K$ is a 
1-ended locally finite and  connected CW complex. A {\it ray} in $K$ is a continuous map $r:[0,\infty)\to K$. A continuous map $f:X\to Y$ is {\it proper} if for each compact set $C$ in $Y$, $f^{-1}(C)$ is compact in $X$. 
The space $K$ has {\it semistable fundamental group at $\infty$} if any two proper rays in $K$ 
are properly homotopic. We say $K$ has {\it semistable first homology at $\infty$} if any two proper rays $r$ and $s$ in $K$ are properly homologous. 

\begin{remark}\label{REdge} 
In a CW-complex, any proper ray is properly homotopic to a proper edge path ray, so for semistability one only need show proper edge path rays are properly homotopic/homologous. 
\end{remark}

We are only interested in simply connected complexes. In this case we need only consider edge path rays with the same initial vertex $\ast$, and homotopies relative to $\ast$.  By properly homologous, we mean there is a proper map $m:M\to K$ where $M$ is a connected 2-manifold with boundary homeomorphic to $\mathbb R^1=(-\infty,\infty)$, and such that for $t\in [0,\infty)$, $m(t)=r(t)$ and for $t\in (-\infty,0]$,  $m(t)=s(-t)$. 
Suppose  $C_0, C_1,\ldots $ is a collection of compact subsets of a locally finite complex $K$ such that $C_i$ is a subset of the interior of $C_{i+1}$ and $\cup_{i=0}^\infty C_i=K$. If $r:[0,\infty)\to K$ is proper, then $\pi_1 (\varepsilon K,r)$ is (up to pro-isomorphism) the  inverse system of groups:
$$\pi_1(K-C_0,r)\leftarrow \pi_1(K-C_1,r)\leftarrow \cdots$$

When $r$ and $s$ are properly homotopic, $\pi_1(\varepsilon K,r)$ is pro-isomorphic to $\pi_1(\varepsilon K,s)$. When $K$ is 1-ended and has semistable fundamental group at $\infty$, $\pi_1(\varepsilon K,r)$ is independent of $r$ (up to pro-isomorphism)  and is called {\it the fundamental group of the end of $K$}. The inverse limit of $\pi_1(\varepsilon K,r)$ is denoted $\pi_1^\infty(K,r)$ and is called {\it the fundamental group at $\infty$ of $K$}.
The group $H_1 (\varepsilon K)$ is (up to pro-isomorphism) the inverse system:
$$H_1(K-C_0)\leftarrow H_1(K-C_1)\leftarrow \cdots$$

These inverse systems are pro-isomorphic to  inverse systems of groups with epimorphic bonding maps if and only if $K$ has semistable fundamental group at $\infty$ (respectively, semistable first homology at $\infty$).  An inverse system of groups that is pro-isomorphic to one with epimorphic bonding maps is called {\it semistable} or {\it Mettag-Leffler}. 

\begin{remark}\label{R1} 
If $K$ has semistable fundamental group at $\infty$, then abelianizing, it immediately follows that $K$ has semistable first homology at $\infty$. 
\end{remark}

Semistablity is an invariant of proper homotopy type and quasi-isometry type. 
There are a number of equivalent forms of semistability. The equivalence of the conditions in the next theorem is discussed in \cite{CM2}. 

\begin{theorem}\label{ssequiv} (see Theorem 3.2, \cite{CM2}) \label{EquivSS} 
Suppose $K$ is a connected 1-ended  locally finite CW-complex. Then the following are equivalent:
\begin{enumerate}
\item $K$ has semistable fundamental group at $\infty$.
\item Suppose $r:[0,\infty )\to K$ is a proper base ray. Then for any compact set $C$, there is a compact set $D$ such that for any third compact set $E$ and loop $\alpha$ based on $r$ and with image in $K-D$, $\alpha$ is homotopic $rel\{r\}$ to a loop in $K-E$, by a homotopy with image in $K-C$. 
\item For any compact set $C$ there is a compact set $D$ such that if $r$ and $s$ are proper rays based at $v$ and with image in $K-D$, then $r$ and $s$ are properly homotopic $rel\{v\}$, by a proper homotopy in $K-C$. 
\end{enumerate}
\end{theorem}

Removing the base rays and replacing $\pi_1$ by $H_1$ in Theorem \ref{ssequiv} gives the corresponding homology result. 
As a consequence of these homology results we have the following result which we will use in \S\ref{T1}:

\begin{theorem} \label{H1SS} 
Suppose $K$ has semistable first homology at $\infty$. If $C$ is a compact subset of $K$ then there is a compact set $D$ of $K$ such that if $\alpha$ is a proper map of $\mathbb R^1$ or the circle $S^1$ into $K-D$ then there is a manifold $M$ with boundary a line/circle, and a proper map $H:M\to K-C$ such that $H$ restricted to $\partial M$ agrees with $\alpha$. 
\end{theorem}

If $G$ is a finitely presented group and $X$ is  a finite connected complex with $\pi_1(X)=G$ then $G$ has {\it semistable fundament group at} $\infty$ (respectively, {\it semistable first homology at $\infty$}) if the universal cover of $X$ has semistable fundamental group (respectively, first homology) at $\infty$. This definition only depends on $G$ and it is unknown if all finitely presented groups have semistable fundamental group at $\infty$.  By Remark \ref{R1}, if $G$ has semstable fundamental group at $\infty$, then $G$ has semistable first homology at $\infty$. 

The following is a partial statement of a result of M. Mihalik and R. Geoghegan. Here $\tilde X$ is the universal cover of $X$. 
\begin{theorem}\label{GeoM} (\cite{GM85},  Theorem 1.1) 
Let $G$ be a group of type $F(n)$ and let $X$ be a $K(G,1)$ CW-complex having finite $n$-skeleton.

(i) For $k\leq n$, $H^k(G,\mathbb ZG)$ mod torsion is free abelian if and only if $H_{k-1}(\varepsilon \tilde X^n)$ is semistable. 

(ii) For $k\leq n$, $H^k(G,\mathbb ZG)$ is torsion free if and only if  $H_{k-2}(\varepsilon \tilde X^n)$ is pro-torsion free. 
\end{theorem}   

When $n=2$ observe that $H_0(\varepsilon \tilde X^n)$ is trivially an inverse system with epimorphic bonding maps (semistable). We immediately deduce:

\begin{corollary}\label{Free} 
If $G$ is a finitely presented group, then $H^2(G,\mathbb ZG)$ is free abelian if and only if $G$ has semistable first homology at $\infty$. 
\end{corollary}

A space $X$ is {\it simply connected at $\infty$} if for each compact set $C$ in $X$ there is a compact set $D$ in $X$ such that loops in $X-D$ are homotopically trivial in $X-C$. If $X$ is simply connected at $\infty$ then $X$ has semistable fundamental group at $\infty$ and in fact, $\pi_1(\varepsilon X,r)$ is pro-trivial for any proper ray $r$. Lemma \ref{HoroSC} shows that certain ``horoballs" are simply connected at $\infty$, an important fact in our proof of Theorem \ref{GM-SS}.

\section{Relatively Hyperbolic Groups and Cusped Spaces} \label{GMSpace} 

We are only interested in locally finite 2-complexes in this paper and we only need define what it means for the 1-skeleton of such a complex to be hyperbolic. There are a number of equivalent definitions of hyperbolicity for a geodesic metric space and it is convenient for us to use the following {\it thin triangles} definition (see  [Definition 1.5, \cite {ABC91}]).

\begin{definition} \label{Hyp} 
Suppose $\Gamma$ is a locally finite 1-complex with edge path metric $d$. Suppose $T=\bigtriangleup(x_1, x_2,x_3)$ is a geodesic triangle (with vertices $x_1$, $x_2$ and $x_3$) in $\Gamma$. 
By inscribing maximal circles in Euclidean comparison triangles (in $\mathbb R^2$), there is a point $c_i$ on the side of $T$ opposite $x_i$ such that $d(x_1,c_3)=d(x_1,c_2)$ and similarly for $x_2$ and $x_3$ (see Figure 1). The points $c_1$, $c_2$ and $c_3$ are called the {\it internal points} of $T$. 

Let $t\in [x_i,c_j]$ and $s\in [x_i,c_k]$ (where $i$, $j$ and $k$ are distict)  be such that $d(x_i,t)=d(x_i,s)$.  If there is a number $\delta\geq 0$ such that  $d(t,s)\leq \delta$ for every geodesic triangle $T$ in $\Gamma$ and all such $t$ and $s$, then $\Gamma$ is (Gromov) $\delta${\it -hyperbolic}. In particular $d(c_i,c_j)\leq \delta$ for all $i,j$. As a notational matter, we let $[a,b]$ be an arbitrary geodesic between the points $a,b\in\Gamma$. 
\end{definition}

\vspace {.3in}
\vbox to 2in{\vspace {-2in} \hspace {-.1in}
\hspace{-.6 in}
\includegraphics[scale=1]{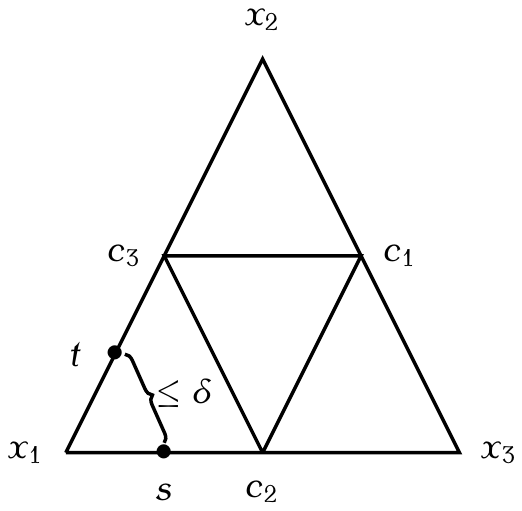}
\vss }

\vspace{-.2 in}

\centerline{Figure 1}

\medskip

D. Groves and J. Manning \cite{GMa08} consider a locally finite Gromov hyperbolic space $X$ constructed from a finitely generated group $G$ and a collection $\mathcal P$ of finitely generated subgroups. The following definitions are directly from \cite{GMa08}

\begin{definition}  Let $\Gamma$ be any 1-complex. The {\it combinatorial horoball} based on $\Gamma$,
denoted $\mathcal H(\Gamma)$, is the 2-complex formed as follows:

{\bf A)} $\mathcal H^{(0)} =\Gamma (0) \times (\{0\}\cup \mathbb N)$

{\bf B)} $\mathcal H^{(1)}$ contains the following three types of edges. The first two types are
called horizontal, and the last type is called vertical.

(B1) If $e$ is an edge of $\Gamma$ joining $v$ to $w$ then there is a corresponding edge
$\bar e$ connecting $(v, 0)$ to $(w, 0)$.

(B2) If $k > 0$ and $0 < d_{\Gamma}(v,w) \leq 2^k$, then there is a single edge connecting
$(v, k)$ to $(w, k)$.

(B3) If $k\geq 0$ and $v\in \Gamma^{(0)}$, there is an edge joining $(v,k)$ to $(v,k+1)$.

{\bf C)} $\mathcal H^{(2)}$ contains three kinds of 2-cells:

(C1) If $\gamma \subset  \mathcal  H^{(1)}$ is a circuit composed of three horizontal edges, then there
is a 2-cell (a horizontal triangle) attached along $\gamma$.

(C2) If $\gamma \subset \mathcal H^{(1)}$ is a circuit composed of two horizontal edges and two
vertical edges, then there is a 2-cell (a vertical square) attached along $\gamma$. 

(C3) If $\gamma\subset  \mathcal H^{(1)}$ is a circuit composed of three horizontal edges and two
vertical ones, then there is a 2-cell (a vertical pentagon) attached along $\gamma$, unless $\gamma$ is the boundary of the union of a vertical square and a horizontal triangle.
\end{definition}

\begin{definition} Let $\Gamma$ be a graph and $\mathcal H(\Gamma)$ the associated combinatorial horoball. Define a {\it depth function}
$$\mathcal D : \mathcal H(\Gamma) \to  [0, \infty)$$
which satisfies:

(1) $\mathcal D(x)=0$ if $x\in \Gamma$,

(2) $\mathcal D(x)=k$ if $x$ is a vertex $(v,k)$, and

(3) $\mathcal D$ restricts to an affine function on each 1-cell and on each 2-cell.
\end{definition}

\begin{definition} Let $\Gamma$ be a graph and $\mathcal H = \mathcal H(\Gamma)$ the associated combinatorial horoball. For $n \geq 0$, let $\mathcal H_n \subset \mathcal H$ be the full sub-graph with vertex set $\Gamma ^{(0)} \times \{0,\ldots ,N\}$, so that $\mathcal H_n=\mathcal D^{-1}[0,n]$.  Let $\mathcal H^n=\mathcal D^{-1}[n,\infty)$ and $\mathcal H(n)=\mathcal D^{-1}(n)$.
\end{definition}

\begin{lemma} \label{GM3.10} 
(\cite{GMa08}, Lemma 3.10) Let $\mathcal H(\Gamma)$ be a combinatorial horoball. Suppose that $x, y \in \mathcal H(\Gamma)$ are distinct vertices. Then there is a geodesic $\gamma(x, y) = \gamma(y, x)$ between $x$ and $y$  which consists of at most two vertical segments and a single horizontal segment of length at most 3.

Moreover, any other geodesic between $x$ and $y$ is Hausdorff distance at most 4 from this geodesic.
\end{lemma}

\begin{definition} Let $G$ be a finitely generated group, let $\mathcal P = \{P_1, \ldots , P_n\}$ be a family of finitely generated subgroups of $G$, and let $S$ be a generating set for $G$ containing generators for each of the $P_i$.   For each $i\in \{1,\ldots ,n\}$, let $T_i$ be a left transversal for $P_i$ (i.e. a collection of representatives for left cosets of $P_i$ in $G$ which contains exactly one element of each left coset).

For each $i$, and each $t \in T_i$, let $\Gamma_{i,t}$  be the full subgraph of the Cayley graph $\Gamma (G,S)$ which contains $tP_i$. Each $\Gamma_{i,t}$ is isomorphic to the Cayley graph of $P_i$ with respect to the generators $P_i \cap  S$. Then define
$$X =\Gamma (G,S)\cup (\cup \{\mathcal H(\Gamma_{i,t})^{(1)} |1\leq i\leq n,t\in T_i\}),$$
where the graphs $\Gamma_{i,t} \subset  \Gamma(G,S)$ and $\Gamma_{i,t} \subset  \mathcal H(\Gamma_{i,t})$ are identified in the obvious way. We call the space $X$ a {\it cusped} space for  $G$, $\mathcal P$ and $S$.  
\end{definition}

The next result shows cusped spaces are fundamentally important spaces.  We prove our results in these spaces. 

\begin{theorem} \label{GM3.25} 
(\cite{GMa08}, Theorem 3.25)
Suppose that $G$ is a finitely generated group and $\mathcal P=\{P_1,\ldots, P_n\}$ is a finite collection of finitely generated subgroups of $G$. Let $S$ be a finite generating set for $G$ containing generating sets for the $P_i$.  The cusped space $X$ for $G$, $\mathcal P$ and $S$ is hyperbolic if and only if  $G$ is hyperbolic with respect to $\mathcal P$.
\end{theorem}

We make some minor adjustments. Assume $G$ is finitely presented and hyperbolic with respect to the finitely presented subgroups $\mathcal P=\{P_1,\ldots, P_n\}$. Take a finite presentation  $\mathcal A$ for $G$ that contains finite presentations for each of the $P_i$ as a subpresentation. Let $S$ be the finite generating set for this presentation. Let $Y$ be the Cayley 2-complex for $\mathcal A$. So $Y$ is simply connected with 1-skeleton $\Gamma(G,S)$, and the quotient space $G/Y$  has fundamental group $G$. Let $X$ be a cusped  space for $G$, $\mathcal P$ and $S$. Replace the Cayley graph $\Gamma(G,S)$ in $X$ with $Y$ in the obvious way. For $g\in G$ and $i\in\{1,\ldots,n\}$ we call $gP_i$ a {\it peripheral  coset}. The depth functions on the horoballs over the peripheral  cosets extend to $X$. So that
$$ \mathcal D:X\to [0,\infty)$$ 
where $\mathcal D^{-1}(0)=Y$ and for each horoball $H$ (over a peripheral  coset) we have $H\cap \mathcal D^{-1}(m)=H(m)$, $H\cap \mathcal D^{-1}[0,m]=H_m$ and $H\cap \mathcal D^{-1}[m,\infty)=H^m$. We call each $H^m$ an {\it $m$-horoball}.

\begin{lemma} \label{geo} (\cite{GMa08}, Lemma 3.26) 
If the cusped space $X$ is $\delta$-hyperbolic, then the $m$-horoballs of $X$ are convex for all $m\geq \delta$. 
\end{lemma} 

Given two points $x$ and $y$ in a horoball $H$, there is a shortest path in $H$ from $x$ to $y$ of the from $(\alpha, \tau,\beta)$ where $\alpha$ and $\beta$ are vertical and $\tau$ is horizontal of length $\leq 3$. Note that if $\alpha$ is non-trivial and ascending and $\beta$ is non-trivial and decending, then $\tau$ has length either 2 or 3. 

Let $\ast$ be the identity vertex of $Y\subset X$ and $d$ the edge path distance in $X$ (so $d$ is measured in $X^1$, the 1-skeleton of $X$). For $v$ a vertex of $X$ and $K\in [0,\infty)$, let $B^1(v, K)=\{y\in X^1: d(y,v)\leq K\}$ and $B(v,K)=B^1(v,K)$ union all 2-cells of $X$ with boundary a subset of $B^1(v,K)$. We call $B(v,K)$ the ball of radius $K$ about $v$ in $X$. 

\begin{lemma}\label{Trans} 
Suppose $X$ is a cusped space for $(G,\mathcal P)$ and $Y\subset X$ is the (simply connected) Cayley 2-complex of a finite presentation for $G$. Given an integer $K$, there is an integer $N(K)$ such that if $\gamma$ is an edge path loop in  $X$ of length $\leq K$, then $\gamma$ is homotopically trivial in $B(v, N(K))$ for any vertex $v$ of $\gamma$. 
\end{lemma}

\begin{proof}
If $\gamma$ has image in a horoball $\mathcal H$, then by using vertical squares  (with two horizontal edges and two vertical edges), $\gamma$ can be slid up in $\mathcal H$ to a loop $\tau_0$ of length $\leq K$ and in a single level of $\mathcal H$ (the highest level attained by $\gamma$). For $v$ a vertex of $\gamma$, and $w$ a vertex of $\tau_0$, there is a vertical path of length $<K$ from $w$ to a vertex of $\gamma$ and so $d(v,w)<2K$. Then $\gamma$ is homotopic to $\tau_0$  by a homotopy (that only uses vertical squares) with image in $B(v, 2K)$ for any vertex $v$ of $\gamma$.  Using vertical pentagons (with two vertical edges, two horizontal edges of $\tau_0$ and one horizontal edge in a level above $\tau_0$), and at most one vertical square, $ \tau_0$ is homotopic to a loop $\tau_1$ with image  one level above the level of $\tau_0$, and $|\tau_1|\leq {|\tau_0|\over 2}+1$. If $v$ is a vertex of $\gamma$ and $w$ a vertex of $\tau_1$, then $d(v,w)<2K+1$ and $\gamma$ is homotopic to $\tau_1$ in $B(v,2K+1)$. If $k$ is the smallest integer such that $K<2^k$, then $\tau_0$ need only be slid up  at most $k-1$ times to the loop $\tau_{k-1}$ where $|\tau_{k-1}|= 2$ (a trivial loop of the form $(e,e^{-1})$). If $v$ is a vertex of $\gamma$ and $w$ a vertex of $\tau_{k-1}$, then $d(v,w)< 2K+k$, and $\gamma$ is homotopically trivial in $B(v,2K+k)$.

If $\gamma$ does not have image in a horoball, write $\gamma=(\alpha_1, \beta_1, \ldots, \alpha_n, \beta_n)$ where $\alpha_i$ is a maximal subpath of $\gamma$ in $Y$ and $\beta_i$ has image in a horoball. By cyclically permuting the edges of $\gamma$ we may assume that $\alpha_1$ is non-trivial (but $\beta_n$ might be trivial). Since $\beta_i$ begins and ends in $Y$, and has length $<K$ there is an edge path $\beta_i'$ in $Y$ with the same end points as $\beta_i$ and $|\beta_i'|<K2^K$ (using vertical squares and pentagons, push $\beta_i$ down to level 0 and note that each horizontal edge of $\beta_i$ is pushed to an edge path of length $<2^K$). The loops $(\beta_i^{-1},\beta_i')$ have length $<(2^K+1)K$ and by the first part there is a constant $N_1(K)$ satisfying the conclusion of the lemma for such loops. In particular $\alpha$ is homotopic to a loop $\alpha'$ in $Y$ by a homotopy in $B(v,N_1(K)+K)$ for any vertex $v$ of $\alpha$, and $|\alpha'|\leq K^2(2^K+1)$. In $Y$ there are only finitely many edge path loops of a given length up to translation by $G$. Hence there is an integer $N_2(K)$ such that any loop in $Y$ of length $\leq K^2(2^K+1)$ is homotopically trivial in $B(v, N_2(K))$ for any vertex $v$ of that loop. Then $\alpha$ is homotopically trivial in $B(v,N_2(K)+N_1(K)+K)$ for any vertex $v$ of $\alpha$. Choose $N(K)$ to be the larger of $N_2(K)+N_1(K)+K$ and $2K+k$. 
\end{proof}

\section{The Proof of Theorem {\ref{GM-SS}}}\label{T2} 

For the remainder of the paper, we will assume that  $G$ is a finitely presented group hyperbolic relative to $\mathcal P=\{P_1,\ldots ,P_n\}$ a set of 1-ended finitely presented peripheral subgroups (with $G\ne P_i$ for all $i$). We assume  $\partial (G, \mathcal P)$ is locally connected. The space $X$ is the cusped space for the group $(G,\mathcal P)$, and $Y\subset X$ is the Cayley 2-complex for $G$ (corresponding to a presentation for $G$ that contains presentations  for the peripherals as subpresentations). We have that $X$ is $\delta$-hyperbolic and $\partial X=\partial (G,\mathcal P)$. A point $p\in \partial (X)$ is an equivalence class of geodesic edge path rays.
Geodesic rays $r$ and $s$ belong to the same equivalence class $p\in \partial X$ if there is a number $K$ such that $d(r(t),s(t))\leq K$ for all $t\in [0,\infty)$.  We write $r\in [s]=p\in \partial X$. 
The space $\partial X$ is a compact metric space. 

\begin{remark}\label{Follow} 
The only general fact that we use about the boundary $\partial X$ of a $\delta$ hyperbolic space $X$ is:  Two points $x,y\in \partial X$ are ``close" precisely when any two representatives $r_x\in x$ and $r_y\in y$ (with $r(0)=s(0)=\ast$) fellow travel for a long time. More precisely, given an integer $K>0$  there is $\epsilon(K)>0$ such that if $x$ and $y$ are within $\epsilon$ in $\partial X$, then $d(r_x(t),r_y(t))\leq \delta$ for all $t\in [0,K]$ all $r_x\in x$ and all $r_y\in y$.
\end{remark}

\begin{lemma} \label{AlmostExt}  (\cite{MS18}, Lemma 6.3) 
For any vertex $v$ of $Y$ there is a geodesic edge path ray $r_v$ in $X$ such that $r_v(0)=\ast$, and for some $t_v\in [0,\infty)$ we have $d(r_v(t_v),v)\leq \delta$ and $r_v|_{[t_v,\infty)}$ has image in $\mathcal D^{-1}([0,21\delta])$.
\end{lemma}

\begin{remark}\label{N1} 
For $t>0$, Lemma \ref{Trans} implies there is an integer $N_1(t)$ such that any edge path loop  in $X$ of length $\leq 7t+3$ is homotopically trivial in the ball $B(v,N_1(t))$ for every vertex $v$ of the loop. The constant $N_1(\delta)$ is used in the proofs of Lemma \ref{SmallH} and Theorem \ref{GM-SS}. 
\end{remark}

The next lemma implies that any two geodesic edge path rays in $X$ are properly homotopic (Corollary \ref{Geo}), and is an important tool used to prove Theorem \ref{GM-SS}.

\begin{lemma}\label{SmallH} 
Suppose $\tau:[0,1] \to \partial X$ is a path. Let $\tau_t$ be a geodesic edge path ray at $\ast$ representing $\tau(t)$ with $\tau_0=r$ and $\tau_1=s$. Suppose $K>0$ is an integer such that for each $t\in [0,1]$, $d(\tau_t(K), r(K))\leq \delta$ (so that if $\tau$ is a small diameter path, $K$ can be chosen large). Let $\beta$ be an edge path from $r(K)$ to $s(K)$ of length $\leq \delta$. Then $r|_{[K,\infty)}$ is properly homotopic rel $\{r(K)\}$ to $(\beta, s|_{[K,\infty)})$ by a homotopy in $X-B(\ast, K-N_1(\delta))$.
 \end{lemma}
 \begin{proof}
 Choose points $q(0)=0<q(1)<\ldots <q(n)=1$ such that for each $i$, the diameter of $\tau([q(i),q(i+1)])$ is small enough to ensure that for each $t\in [q(i), q(i+1)]$ and $k\in [0,2K]$, $d(\tau_t (k), \tau_{q(i)}(k))\leq \delta$. In particular, 
 $$d(\tau_{q(i)}(k), \tau_{q(i+1)}(k)\leq \delta \hbox{ for all } k\in [0,2K]$$

Write the consecutive vertices of $\tau_{q(i)}$ as $\ast=v_0, v_1,\ldots$  and those of $\tau_{q(i+1)}$ as $\ast=w_0,w_1,\ldots$. Note that $d(v_j,w_j)\leq \delta$ for all $j\leq 2K$. Let  $[v_K,w_K]=\gamma_i$, and $[v_{2K},w_{2K}]=\beta_i$ be geodesic edge paths of length $\leq \delta$. For ease of notation, let $[v_{j-1},v_{j}] =\tau_{q(i)}|_{[j-1,j]}$ and $[w_{j-1},w_{j}] =\tau_{q(i+1)}|_{[j-1,j]}$.

\vspace {.5in}
\vbox to 2in{\vspace {-2in} \hspace {-.6in}
\hspace{-.6 in}
\includegraphics[scale=1]{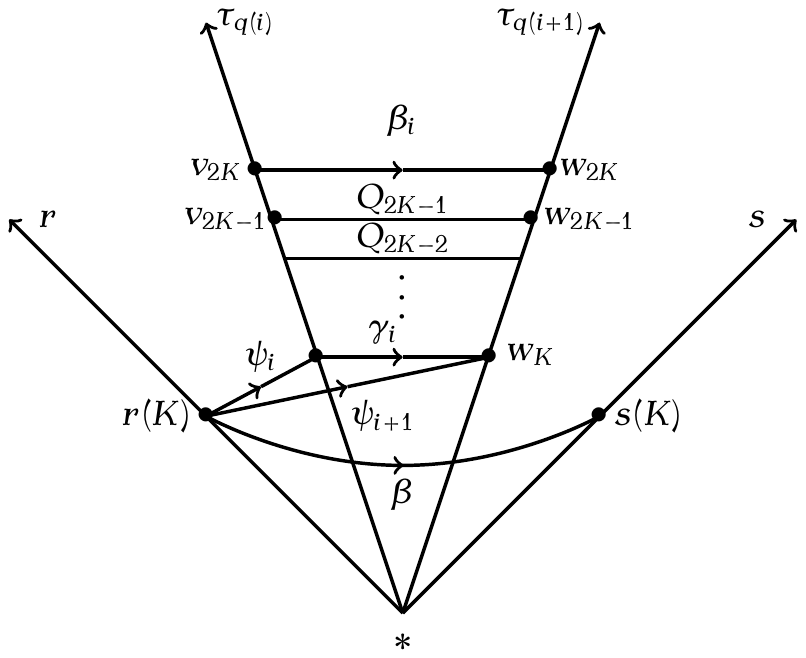}
\vss }

\vspace{.3 in}

\centerline{Figure 2}

\medskip

The geodesic quadrilaterals $Q_j=([v_j,v_{j+1}], [v_{j+1}, w_{j+1}], [w_{j+1}, w_j], [w_j,v_j]$ have an edge on $\tau_{q(i)}$ an edge on $\tau_{q(i+1)}$ and for $j\leq 2K-1$ boundary length $\leq 2\delta+2$. So for $j\leq 2K-1$,  $Q_j$ is homotopically trivial by a homotopy in $B(v,N_1(\delta))$ for any vertex $v$ of $Q_j$  (see Remark \ref{N1}). In particular, for $j\in \{K,\ldots, 2K-1\}$,  $Q_j$ is homotopically trivial in $X-B(\ast, K-N_1(\delta))$ (see Figure 2).

The rectangle $R_i$ bounded by $\tau_{q(i)} ([K,2K])$, $\tau_{q(i+1)} ([K,2K])$, $\beta_i$ and $\gamma_i$ is subdivided by the $Q_j$ for $j\in \{K,\ldots, 2K-1\}$.  Hence $R_i$ is homotopically trivial in $X-B(\ast, K-N_1(\delta))$.
Combining the null homotopies for the $R_i$, we have that the rectangle determined by $r([K,2K])$, $s([K,2K])$, $(\beta_0,\ldots, \beta_{n-1})$ and $(\gamma_0,\ldots, \gamma_{n-1})$ is homotopically trivial by a homotopy $H_1'$, that avoids $B(\ast, K-N_1(\delta))$.
  
Next let $\psi_i$ be an edge path of length $\leq \delta$ from $r(K)$ to $\tau_{q(i)}(K)$. Each of the edge path loops $(\gamma_0, \gamma_1,\psi_2^{-1})$, $(\psi_i, \gamma_i, \psi_{i+1}^{-1})$, $(\psi_{n-1},\gamma_{n-1}, \beta^{-1})$ has length $\leq 3\delta$ and is homotopically trivial by a homotopy avoiding $B(\ast, K-N_1(\delta))$.   
Combining these homotopies we have $(\gamma_0,\ldots, \gamma_{n-1})$ is homotopic to $\beta$ by a homotopy avoiding $B(\ast, K-N_1(\delta))$. Combining this homotopy with $H_1'$ we have a null homotopy $H_1$ of the loop determined by $r([K,2K])$, $s([K,2K])$, $(\beta_0,\ldots, \beta_{n-1})$ and $\beta$, such that $H_1$ avoids $B(\ast, K-N_1(\delta))$.

Note that each $\beta_i$ has image avoiding $B(\ast, 2K-{\delta\over 2})$. For each $i$, letting $\beta_i$, $\tau_{q(i)}$ and $\tau_{q(i+1)}$ play the role of $\beta$, $r$ and $s$ respectively and $\tau|_{[q(i),q(i+1)]}$ the role of $\tau$,  a completely analogous argument shows that the edge path $((r([2K,3K]))^{-1}, \beta_0,\ldots , \beta_{n-1}, s([2K,3K]))$ is homotopic to a path with image avoiding $B(\ast, 3K-{\delta\over 2})$ by a homotopy $H_2$ with image avoiding the ball $B(\ast, 2K-N_1(\delta))$. Patching together the $H_i$ gives the desired proper homotopy. 
\end{proof}

\begin{corollary} \label{Geo} 
Suppose $r$ and $s$ are geodesic edge path rays in $X$ and $\beta$ is an edge path from $r(0)$ to $s(0)$ then $r$ is properly homotopic to $(\beta, s)$ rel$\{r(0)\}$. 
\end{corollary}
\begin{proof}
Let $r_1$ and $s_1$ be geodesic edge path rays at $\ast$ such that $r\in [r_1]$ and $s\in [s_1]$. Let $\gamma_r$ and $\gamma_s$ be arbitrary edge paths from $r(0)$ to $\ast$ and $s(0)$ to $\ast$ respectively. Certainly $r$ is properly homotopic to $(\gamma_r,r_1)$ rel$\{r(0)\}$ and $s$ is properly homotopic to $(\gamma_s,s_1)$ rel$\{s(0)\}$. By Lemma \ref{SmallH}, $r_1$ is properly homotopic to $s_1$ rel$\{\ast\}$. (Apply Lemma \ref{SmallH} with $K=0$ and $\tau$ any path in $\partial X$ from $[r_1]$ to $[s_1]$. Also note  that $B(\ast, -N_1(\delta))=\emptyset$.) Since $X$ is simply connected, the loop  $(\beta,\gamma_s, \gamma_r^{-1})$ is homotopically trivial. Simply combine the four homotopies. 
\end{proof}  

\begin{lemma}\label{HoroSC}
For any group $P$ with finite generating set $S$, the horoball for $(P,S)$ is simply connected at $\infty$. 
\end{lemma}
\begin{proof}
Let $\mathcal H$ be the horoball for $(P,S)$ and $C$ a compact subcomplex of $\mathcal H$. Let $D$ be the full subcomplex containing $C$ and all vertices ``between $C$ and level-0" (so if $v$ is a vertex of $C$ and $l_v$ is the vertical line thru $v$, then $D$ contains all vertices of $l_v$ from level 0, to the level of $v$). So if $v$ is a vertex of $\mathcal H-D$, then the vertical line at $v$ avoids $D$. Let $\alpha$ be an edge path loop in $\mathcal H-D$. Then using vertical squares, $\alpha$ can be ``slid"  directly up to an edge path $\alpha_1$ in a single level above the top level of $D$, by a homotopy avoiding $D$. Write $\alpha_1$ as the edge path loop $(e_1,e_2,\ldots, e_n)$. Using vertical pentagons (and perhaps one vertical square), each of the pairs of edges $(e_1, e_2), (e_3,e_4),\ldots$ can be slid up to edges $(d_1,\ldots, d_m)$ respectively, where $m\leq [{n\over 2}]+1$. Continuing this process, one ends up with an edge path loop of length $3$, which is homotopically trivial in that level. Combining homotopies, $\alpha$ is homotopically trivial in $\mathcal H-C$ (actually in $\mathcal H-D$).
\end{proof}

\begin{lemma}\label{HoroRay} 
Suppose $r$ is a proper edge path ray at $\ast\in Y$, with image in $X$ such that no tail of $r$ has image in a horoball, then $r$ is properly homotopic to an edge path ray in $Y$. 
\end{lemma}
\begin{proof}
Suppose $\mathcal H$ is a horoball. If $r$ has only finitely many edges in $\mathcal H$, then let $\alpha_1,\ldots, \alpha_n$ be the maximal  subpaths of $r$ that begin and end in $Y$ and have image in $\mathcal H$. Let $\beta_i$ be an edge path in $Y\cap \mathcal H$ with the same initial and end point as $\alpha_i$. Since $ \mathcal H$ is simply connected, $\alpha_i$ and $\beta_i$ are homotopic in $ \mathcal H$, relative to their common end points. Let $s$ be the proper  edge path ray obtained from $r$ by replacing the $\alpha_i$ by the $\beta_i$ for {\it every} horoball $ \mathcal H$ such that $r$ meets $ \mathcal H$ in only finitely many edges. There is an obvious homotopy $K$ from $r$ to $s$ and $K$ is proper since a compact set  in $X$ can only intersect finitely many horoballs.

Let $\Gamma_i$ be the Cayley graph of $P_i$ with respect to the generating set used here. Suppose $ \mathcal H$ is a horoball and $r$ has infinitely many edges in $ \mathcal H$. Let $\alpha_1, \alpha_2,\ldots$ be the maximal subpaths of $r$ that begin and end in $Y$ and have image in $\mathcal H$. Assume that the $\alpha_i$ are ordered as they appear as subpaths of $r$. Note that the $\alpha_i$ are also ordered subpaths of $s$. Let $C_0\subset C_1\subset\cdots$ be a collection of compact subsets of $X$ such that $C_i$ is a subset of the interior of $C_{i+1}$ and $\cup_{i=1}^\infty C_i=X$. Since each $P_i$ is 1-ended, we may  assume  that if $\mathcal H$ is the  horoball corresponding to the coset $gP_i$ (so that $\mathcal H\cap Y=g\Gamma_i$) and $\mathcal H\cap C_j\ne\emptyset$ then $C_{j+1}$ contains all bounded components of  $g\Gamma_i-C_j$. By Lemma \ref{HoroSC}, $\mathcal H$ is simply connected at infinity and so we may assume that if $\gamma$ is a loop in $\mathcal H-C_{j+1}$ then $\gamma$ is homotopically trivial in $\mathcal H-C_j$. For convenience let $C_i=\emptyset$ for $i\leq 0$. For each $k\geq 1$, let $j(k)\geq 0$ be the largest integer such that $\alpha_k$ has image in $\mathcal H-C_{j(k)}$. Then there is an edge path $\beta_k$ in $g\Gamma_i-C_{j(k)-1}$ with the same initial and end point as $\alpha_k$. Since $\mathcal H$ is simply connected at infinity, $\alpha_k$ and $\beta_k$ are homotopic relative to their endpoints by a homotopy in $\mathcal H-C_{j(k)-2}$. 

Again, any compact set $C$ intersects only finitely many horoballs. Given any horoball $\mathcal H$, only finitely many of the homotopies of the $\alpha_k$ to the $\beta_k$ intersect $C$. Combining homotopies, $s$ (and hence $r$) is properly homotopic to a proper  edge path ray in $Y$. 
\end{proof}

\begin{proof} (of Theorem \ref{GM-SS}) By Corollary \ref{Geo} it is enough to show that each proper edge path ray based at $\ast$ in $X$ is properly homotopic rel$\{\ast\}$, to a geodesic  edge path ray at $\ast$. Suppose $r$ is a proper edge path ray at $\ast$ in $X$ with tail in the horoball $\mathcal H$. Let $z$ be a closest point of $\mathcal H(\delta)$ (the points of $\mathcal H$ in level $\delta$) to $\ast$ and $\alpha$ a geodesic edge path from $\ast$ to $z$. Let $s$ be the vertical geodesic  edge path ray in $\mathcal H$ beginning at $z$. By Lemma \ref{geo}, $(\alpha, s)$ is a geodesic  edge path ray. Let $v$ be the first vertex of $r$ such that each vertex following $v$ belongs to $\mathcal H$. Let $q$ be the vertical geodesic edge path ray at $v$. By pushing horizontal edges up along vertical squares, $q$ and the tail of $r$ at $v$ are properly homotopic rel$\{v\}$. Let $\beta$ be the initial segment of $r$ from $\ast$ to $v$. By Corollary  \ref{Geo}, the  edge path rays $(\alpha, s)$ and $(\beta, q)$ are properly homotopic rel$\{\ast\}$. But then $r$ is properly homotopic to the geodesic  edge path ray $(\alpha,s)$ rel$\{\ast\}$.

\vspace {.5in}
\vbox to 2in{\vspace {-2in} \hspace {-1.3in}
\hspace{-.6 in}
\includegraphics[scale=1]{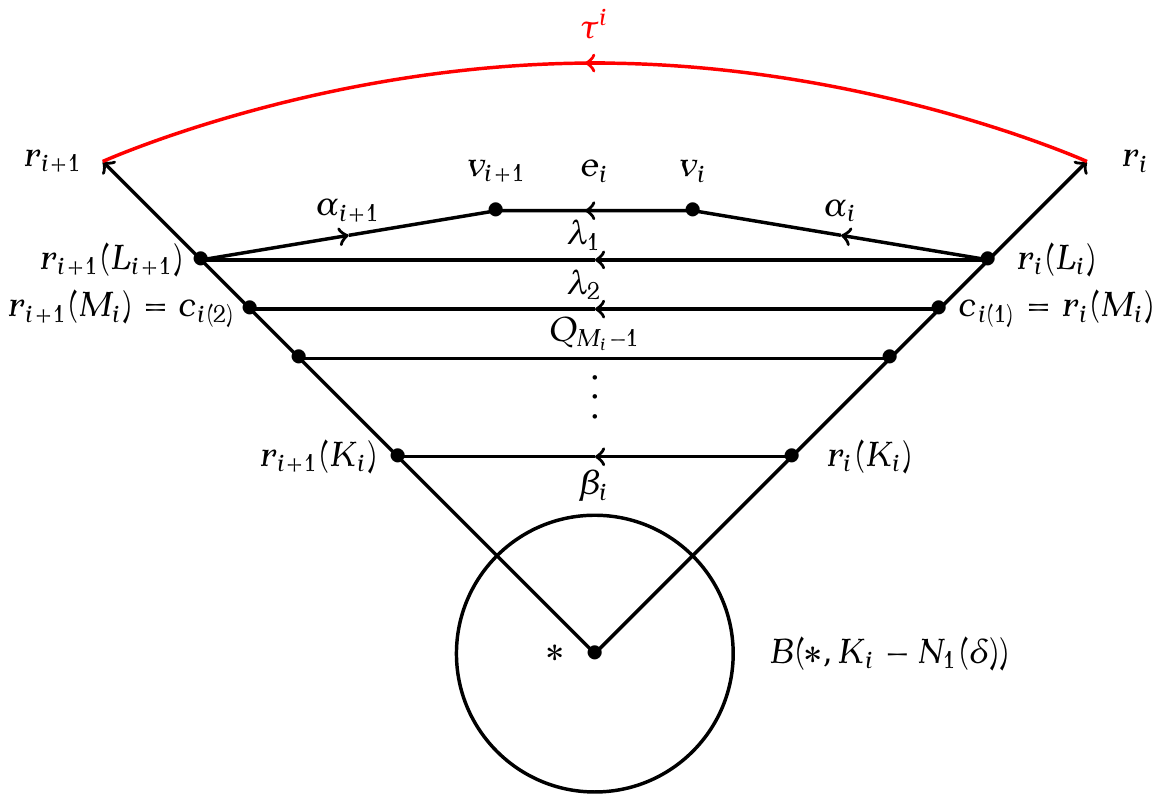}
\vss }

\vspace{.9 in}

\centerline{Figure 3}

\medskip

By Lemma \ref{HoroRay}, we now only need show that  a general proper edge path ray at $\ast$ and with image in $Y$  is properly homotopic rel$\{\ast \}$ to a geodesic  edge path ray at $\ast$ (Corollary \ref{Geo}). 
Suppose $r$ is a proper edge path ray in $Y$ that is based at $\ast$. List the consecutive edges of $r$ as $e_0,e_1,\ldots$ and consecutive vertices as $\ast=v_0, v_1,\ldots$. Let $r_0$ be a geodesic  edge path ray beginning at $v_0$ and for each $i\geq 1$, let $r_i$ be a geodesic edge path ray in $X$ such that $r_i(0)=\ast$ and for some $L_i\in [0,\infty)$, $d(r_i(L_i), v_i)\leq \delta$ (see Lemma \ref{AlmostExt}). Let $\alpha_i$ be an edge path of length $\leq \delta$ from $r_i(L_i)$ to $v_i$ (see Figure 3).

Since $r$ is proper, $\lim_{i\to \infty} \{L_i\}=\infty$. Let $\lambda_1$ be a geodesic from $r_i(L_i)$ to $r_{i+1}(L_{i+1})$. Then $|\lambda_1|\leq 2\delta+1$.
Let $T_i$ be the geodesic triangle with sides $r_i([0,L_i])$, $\lambda_1$ and $r_{i+1}([0,L_{i+1}])$. Let $c_{i(1)}$ and $c_{i(2)}$ be the internal points of $T_i$ on $r_i$ and $r_{i+1}$ respectively. By the definition of internal points: 
$$d(c_{i(1)}, r_i(L_i))+d(c_{i(2)}, r_{i+1}(L_{i+1}))=|\lambda_1|\leq 2\delta+1$$ 
Also, $d(c_{i(1)},c_{i(2)})\leq \delta$. The definition of internal points implies that  if $c_{i(1)}=r_i(M_i)$, then $c_{i(2)}=r_{i+1}(M_i)$. So $L_i-M_i=d(c_{i(1)}, r_i(L_i))\leq 2\delta+1$ (and $L_{i+1}-M_i\leq 2\delta+1$). In particular $\lim_{i\to\infty}\{M_i\}=\infty$. Let $\lambda_2$ be a geodesic (of length $\leq \delta$) from $c_{i(1)}=r_i(M_i)$ to $c_{i(2)}=r_{i+1}(M_i)$. We have:
 $$|(\lambda_2, r_{i+1}|_{[M_i,L_{i+1}]}, (\alpha_{i+1}, e^{-1}_i,\alpha_{i}^{-1}), (r_i|_{[M_i,L_i]})^{-1})|\leq 7\delta +3$$ 
So there is a null homotopy $H_i$ for this loop in $X-B(\ast, M_i-N_1(\delta))$ (see Figures 3 and 4(a)).

\vspace {.5in}
\vbox to 2in{\vspace {-2in} \hspace {-2.3in}
\hspace{.1 in}
\includegraphics[scale=1]{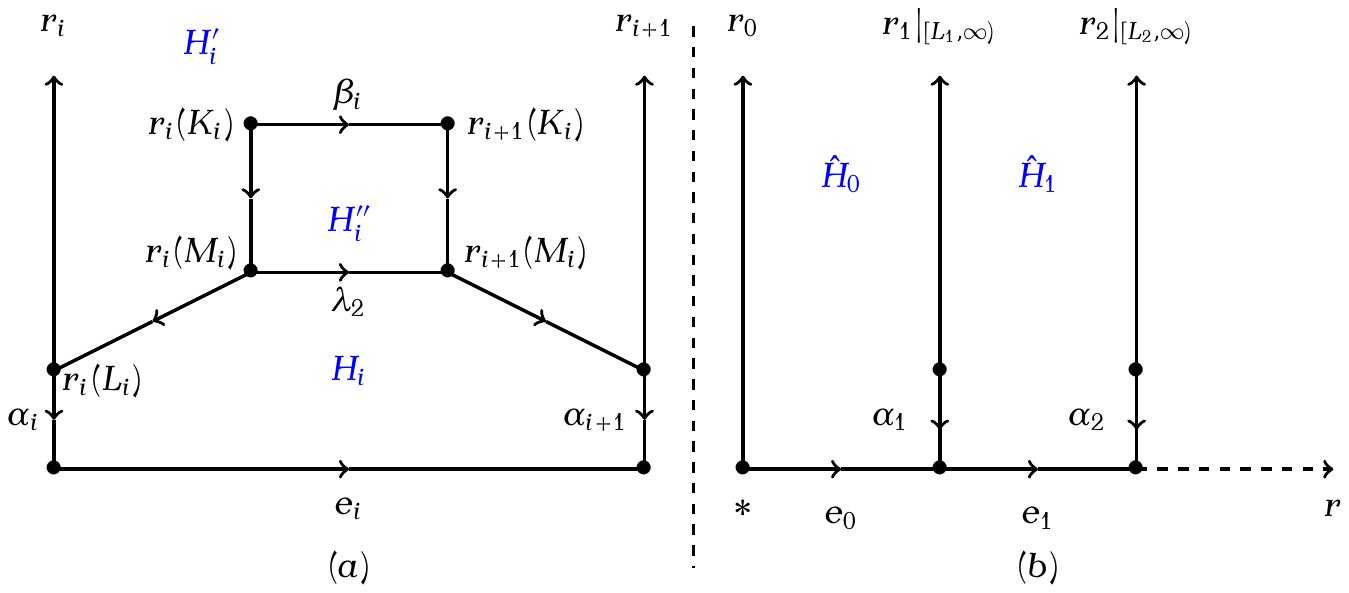}
\vss }

\vspace{-.1 in}

\centerline{Figure 4}

\medskip

For large $i$,  $r_i$ and $r_{i+1}$ fellow travel for a long time. By the local connectivity of $\partial X$, there is a path $\tau^i$ in $\partial X$ of diameter $\epsilon _i$ connecting $[r_i]$ and $[r_{i+1}]$ where $\lim_{i\to \infty} \epsilon_i= 0$. Following the notation of  Lemma \ref{SmallH}, let $\tau^i_t\in [\tau^i(t)]$. Then we may choose $K_i\leq M_i$ such that for all $t\in [0,1]$, $d(\tau^i_t(K_i), r_i(K_i))\leq \delta$, and $\lim_{i\to\infty} \{K_i\}= \infty$. Let $\beta_i$ be a geodesic edge path (of length $\leq \delta$) from $r_i(K_i)$ to $r_{i+1}(K_i)$. 
By Lemma \ref{SmallH}, $r_{i}|_{[K_i,\infty)}$ is properly homotopic $rel\{r_{i}(K_i)\}$ to $(\beta_i, r_{i+1}|_{[K_i,\infty)})$ by a homotopy $H_i'$ with image in $X-B(\ast, K_i-N_1(\delta))$. 

Just as with Lemma \ref{SmallH} the geodesic quadrilateral $R_i$ with sides $\lambda_2$, $\beta_i$, $r_i|_{[K_i,M_i]}$ and $r_{i+1}|_{[K_i,M_i]}$ can be subdivided by geodesic quadrilaterals $Q_j$, where two opposite sides of $Q_i$ are corresponding edges of $r_i$ and $r_{i+1}$ ($r_i([k,k+1])$ and $r_{i+1}([k,k+1])$ for some integer $k$) and the other two sides are geodesics of length $\leq \delta$ (see Figures 2 and 3). 
Since each $Q_i$ has boundary path of length $\leq 2\delta+2$, it is homotopically trivial in $X-B(\ast, K_i-N_1(\delta))$, and so $R_i$ is null homotopic by a homotopy $H_i''$ with image in $X-B(\ast, K_i-N_1(\delta))$. Combining the homotopies $H_i$, $H_i'$ and $H_i''$, (as in Figure 4(a)) $r_i|_{[L_i,\infty)}$ is properly homotopic to $(\alpha_i, e_i, \alpha_{i+1}^{-1},r_{i+1}|_{[L_{i+1},\infty)})$ by a homotopy $\hat H_i$ with image in $X-B(\ast, K_i-N_1(\delta))$. 

Since $r_0$ begins at $\ast$, we may assume that $\alpha_0$ is trivial. Patching together the homotopies $\hat H_i$ (for $i\in \{0,1,\ldots\}$) we have a homotopy $\hat H$ of $r_0$ to $r$ (see Figure 4 (b)). For any compact set $C$, only finitely may of the $\hat H_i$ have image that intersect $C$ and so $\hat H$ is proper.
\end{proof}

\section{Proper Relative Simplicial Approximation (SimpApp)}\label{SimpApp} 

In this section $X$ is a simplicial complex and $Y$ will be a certain subcomplex. In our applications of the main result of this section, again $X$ will be the cusped space for $(G,\mathcal P)$ and $Y$ will be a Cayley 2-complex for $G$. 

The spaces $[0,\infty)\times [0,\infty)$ and $[0,\infty) \times [0,1]$ are homeomorphic. The proper maps $r,s:([0,\infty),\{0\})\to (X,\ast)$ are properly homotopic $rel\{\ast\}$ means there is a proper homotopy $H:[0,\infty) \times [0,1]\to X$ such that $H(0,t)=r(t)$, $H(1,t)=s(t)$ and $H(0,t)=\ast$ for all $t$. This is equivalent to having a proper homotopy  $H':[0,\infty)\times [0,\infty)\to X$ such that $H'(t,0)=r(t)$ and $H'(0,t)=s(t)$ for all $t$. In this section we use the space $[0,\infty)\times [0,1]$ for technical reasons.

All spaces are simplicial complexes, and all subcomplexes are full subcomplexes of the over complex. If $X$ is a simplicial complex, then we say a subcomplex $Z$ {\it separates} a vertex $v\in X-Z$ from a subcomplex $Y$ if any edge path in $X$ from $v$ to a vertex of $Y$ contains a vertex of $Z$. 

Our primary models for the main theorem of this section is when the space $X$ is the Cayley 2-complex for a group split non-trivially as $G=A\ast_CB$ where $A$ and $B$ are finitely presented and $C$ is finitely generated, or $X$ is a cusped space for a finitely generated group  which is hyperbolic relative to a finite collection of proper finitely generated subgroups. In this article we are only interested in the cusped space version, but our result will have important uses elsewhere. 

We are interested in proper homotopies 
$$M:[0,\infty)\times [0,1]\to  X$$
 of proper edge path rays  $r$ and $s$ into a connected locally finite simplicial 2-complex $X$, where $r$ and $s$ have image in a subcomplex $Y$ of $X$. Simplicial approximation allows us to assume that $M$ is simplicial (see Lemma \ref{simpA}). In the case $G$ is the amalgamated product $A\ast_CB$, the space $X$ is the Cayley 2-complex for $G$, and the space $Y$ is the Cayley subcomplex of $X$ for $A$. When $G$ is hyperbolic relative to the finite set $\mathcal P$ of proper finitely generated subgroups of $G$, the space $X$ is the cusped space for $(X,\mathcal P)$ and $Y$ is the Cayley 2-complex for $G$.  Say $\{Z_i\}_{i=1}^\infty$ is a collection of connected subcomplexes of $Y$ such that only finitely many $Z_i$ intersect any compact subset of $X$, and each vertex of $X-Y$ is separated from $Y$ by exactly one $Z_i$. In the $A\ast_CB$ setting, the $Z_i$ correspond to $aC$ cosets for each $a\in A$. In the relative hyperbolic setting the $Z_i$ correspond to $gP$ where $P\in \mathcal P$ and  $g\in G$. Assume that each vertex of the image of $M$ is either in $Y$ or is separated from $Y$ by some $Z_i$. 

The next result describes how to simplicially excise the interiors of certain subcomplexes of $[0,\infty)\times [0,1]$ which together contain all points that $M$ does not mapped into $Y$. What is removed from  $[0,\infty)\times [0,1]$ is a disjoint union of open sets $E_j$ for $j\in J$, each homeomorphic to $\mathbb R^2$ and $M([0,\infty)\times [0,1]-(\cup_{j\in J} E_j))\subset Y$. Also, $M$ maps the topological boundary of each $E_j$ into some $Z_i$.  When $E_j$ is bounded (contained in a compact set) in $[0,\infty)\times [0,1]$,  there is an embedded edge path loop $\alpha_j$ in $[0,\infty)\times [0,1]$ that bounds $E_j$, and $M(\alpha_j)$ has image in one of the $Z_i$. In this case the $E_j$ and $\alpha_j$ form a finite subcomplex of $[0,\infty)\times [0,1]$ homeomorphic a closed ball which contains a certain equivalence class of triangles that are mapped into $X-Y$. When $E_j$ is unbounded, $\alpha_j$ is an embedded proper edge path line that bounds $E_j$. Again $M(\alpha_j)$ has image in one of the $Z_i$ and $E_j$ and $\alpha_j$ form a subcomplex of $[0,\infty)\times [0,1]$ homeomorphic to the closed upper half plane.  Again, in this case $E_j$ will contain a certain equivalence class of triangles, each of which is mapped into $X-Y$. For our purposes in this article, we intend to replace each $E_j$ by a 2-manifold with boundary a circle or real line and define a map of the 2-manifold into $Y$ so that the new map on the boundary of the 2-manifold is basically $\alpha_j$. Attaching these 2-manifolds to  $[0,\infty)\times [0,1]-(\cup_{j\in J} E_j)$ along the boundaries of the $E_i$ will produce a 2-manifold that will show $r$ and $s$ are properly homologous. 
 
\begin{definition}
We call the pair $(E,\beta)$ a {\it disk pair} in the simplicial complex $[0,\infty)\times [0,1]$ if $E$ is an open subset of $[0,\infty)\times [0,1]$ homeomorphic to $\mathbb R^2$, $E$ is a union of (open) cells, $\alpha$ is an embedded edge path bounding $E$ and $E$ union $\alpha$ is a closed subspace of $[0,\infty)\times [0,1]$ homeomorphic to a closed ball or a closed half space in $[0,\infty)\times [0,1]$. When $\alpha$ is finite, we say the disk pair is finite, otherwise we say it is unbounded. Note that if $(E,\beta)$ is a disk pair, then $\beta$ is collared in $E$.
\end{definition}

\begin{theorem} \label{excise} 
Suppose $M:([0,\infty)\times [0,1],\{0\}\times [0,1])\to  (X,\ast)$ is a proper simplicial homotopy $rel\{\ast\}$ of proper edge path rays  $r$ and $s$ into a connected 
locally finite simplicial 2-complex $X$, where $r$ and $s$ have image in a subcomplex $Y$ of $X$. Say $\mathcal Z=\{Z_i\}_{i=1}^\infty$ is a collection of connected subcomplexes of $Y$ such that only finitely many $Z_i$ intersect any compact subset of $X$. Assume that each vertex of $X-Y$ is separated from $Y$ by exactly one $Z_i$. 

Then there is an index set $J$ and for each $j\in J$, there is a disk pair $(E_j,\alpha_j)$ in  $[0,\infty)\times [0,1]$ where the $E_j$ are disjoint, $M$ maps $\alpha_j$  to $Z_{i(j)}$ (for some $i(j)\in \{1,2,\ldots\}$) and $M([0,\infty)\times [0,1]-\cup_{j\in J} E_j)\subset Y$. 
\end{theorem} 

\begin{remark}\label{Tfacts} 
Assume the hypotheses of Theorem \ref{excise}. If $e$ is an edge of $[0,\infty)\times [0,1]$ then there are at most two $j\in J$   such that $e$ is an edge of $\alpha_j$. Hence if $K$ is a finite subcomplex of $[0,\infty)\times [0,1]$, there are only finitely many $j\in J$ such that $\alpha_j$ has an edge in $K$. 

If $Z_i\in \mathcal Z$ is a finite subcomplex of $X$,  and $M$ maps $\alpha_j$ to $Z_i$ then since $M$ is proper, $\alpha_j$ is a circle (and not a line), so that $E_j$ is bounded in $[0,\infty)\times [0,1]$. Also since $M$ is proper,  $M^{-1}(Z_i)$ is compact, and so $M$ maps only finitely many $\alpha_j$ to $Z_i$.  
\end{remark}

\begin{proof}
For each  $j\in \{1,2,\ldots\}$ we define a coloring $\mathcal C_j$ of $X$ using only three colors. Basically cells on one side of $Z_j$ are one color, cells on the other side of $Z_j$ are a different color and the cells of $Z_i$ are a third color. 

All vertices, open edges and open triangles of $Z_j$ are colored blue.  All vertices open edges and open triangles of $X-Y$ that are separated from $Y$ by $Z_j$ are colored red (so any edge path from a red vertex to a vertex of $Y$ contains a blue vertex). Each vertex, open edge and open triangle that remains is colored green.   A red triangle may be bounded by both red and blue edges. 
The following are elementary, but critical observations for each coloring $C_j$:

{\bf (1)} An edge $e$ is blue if and only if $e$ has two blue vertices and a triangle $\triangle$ is blue if and only if $\triangle$ has three blue vertices. 

{\bf (2)} An edge $e$ has a red vertex if and only if $e$ is red. Every red edge has either two red vertices or one red vertex and one blue vertex. A vertex (edge) of a red triangle is either red or blue. A triangle is red if and only if it has at least one red vertex (and so at least two red edges). If a triangle has only one red vertex $v$, then the two edges adjacent to $v$ are red and the edge opposite $v$ is blue. If a triangle has exactly two red vertices, then all three edges are red and the third vertex is blue. Finally a red triangle might have three red vertices and three red edges. 

Giving a coloring $\mathcal C_j$ of $X$, color the vertices, edges and triangles of $[0,\infty)\times [0,1]$ by the color of their image simplex. An elementary check shows observations {\bf (1)} and {\bf (2)} are true for simplices of $[0,\infty)\times [0,1]$. We want to cut out the red part of $[0,\infty)\times [0,1]$ for all colorings $\mathcal C_i$. 
If for the coloring $\mathcal C_i$, there is no red vertex $v$ of $[0,\infty)\times [0,1]$, then eliminate $Z_i$ from $\mathcal Z$.   

Let $S_i$ be the set of triangles $\triangle$ of $[0,\infty)\times [0,1]$ such that $M(v)$ is separated from $Y$ by $Z_i$ for some vertex $v$ of $\triangle$. So $S_i$ is the set of red triangles of $[0,\infty)\times [0,1]$ with respect to $\mathcal C_i$. 
Each vertex of $X-Y$ is separated from $Y$ by exactly one $Z_i$, so if $i\ne j$ then $S_i\cap S_j=\emptyset$. 

Next we show that if $Z_i$ is a finite complex then $S_i$ is a finite set. Otherwise there are triangles of $S_i$ outside of any given compact subset of $[0,\infty)\times [0,1]$. But $M^{-1} (Z_i)$ is compact and so contained in $[0,K]\times [0,1]$ for some integer $K$. Select $\triangle\in S_i$  with image in $[K+1,\infty)\times [0,1]$. Let $w$ be a vertex of $\triangle$ and $\alpha$ a path in $[K+1,\infty)\times [0,1]$ from $w$ to $(K+1,0)$ then $M(\alpha)$ is a path in $X$ from $M(w)$ (a point on the side of $Z_i$ opposite $Y$) to a point of $Y$, and $M(\alpha)$ avoids $Z_i$. This is impossible. Instead:

{\bf (3)} If $Z_i$ is a finite complex, the set $S_i$ is finite.

Next we partition the set $S_i$ into disjoint subsets. If $\triangle, \triangle' \in S_i$ then $\triangle\sim_i \triangle'$ if there is a sequence of  triangles $\triangle=\triangle_1,\triangle_2,\ldots, \triangle_k=\triangle'$ such that (for each $j$) $\triangle_j$ and $\triangle_{j+1}$ share a red (with respect to $\mathcal C_i$) edge. Clearly $\sim_i$ is an equivalence relation on triangles of $S_i$. Let $T_i(\triangle)$ be the set of triangles in $[0,\infty)\times [0,1]$ that are $\sim_i$ equivalent to $\triangle$.

As $M$ maps  $[0,\infty)\times \{0,1\}$ into $Y$, each edge of $[0,\infty)\times \{0,1\}$ is blue or green in any coloring.  If each edge of $[0,\infty)\times \{0,1\}$ is blue in some coloring, then the theorem is trivially true with one $E_j=(0,\infty)\times (0,1)$.

Let  $g_i$ be a green vertex of $[0,\infty)\times \{0,1\}$ with respect to $\mathcal C_i$.

\begin{lemma}\label{partition} 
Suppose  $\triangle\in  S_i$. Then there is a disk pair $(E,\beta)$ in $[0,\infty)\times [0,1]$ such that 

(i) Each vertex and edge of $\beta$ is blue under $\mathcal C_i$ (so $M$ maps $\beta$ to $Z_i$) and each edge of $\beta$ belongs to exactly one triangle of $T_i(\triangle)$.

(ii) Triangle $\triangle$ is in $E$. If $j\in \{1,2,\ldots\}$,  $\triangle'\in S_j$  and $\triangle'$ is in $E$, then each triangle of $T_j(\triangle')$ is in $E$. In particular, the triangles of $T_i(\triangle)$ are all in $E$.  

(iii) If $Z_i$ is finite, then 
$(E,\beta)$ is finite. 

(iv) The edge path $\beta$ separates the vertices of $E$ from $[0,\infty)\times \{0,1\}$.
\end{lemma}
\begin{proof} 
All colors in this proof are with respect to $\mathcal C_i$. Let $T=T_i$.
Let $\triangle_1' $ be a green triangle containing $g=g_i$ and $\triangle_1',\triangle_2',\ldots, \triangle_k'$ be a sequence of triangles such that consecutive triangles share an edge and such that $k$ is the first integer such that $\triangle_k'$ is  in $T(\triangle)$ (so $T(\triangle_k')=T(\triangle)$). Let $e$ be the edge shared by $\triangle_{k-1}'$ and $\triangle_k'$. If $e$ is red, then by {\bf (2)}, $e$ has a red vertex and so $\triangle_{k-1}'$ is red. That is impossible since $\triangle_{k-1}'$ is not in $T(\triangle_k')$. By {\bf (2)}, $e$ is blue.  

Let $v$ be the red vertex (opposite $e$) of  triangle $\triangle_k$ of $[0,\infty)\times [0,1]$. Each triangle of $[0,\infty)\times [0,1]$ containing $v$ is in $T(\triangle)$ and is red. Let $\alpha_1$ be an embedded edge path loop circumventing the boundary of $st(v)$. The edge $e$ is an edge of $\alpha_1$, each  edge of $\alpha_1$ is red or blue and $\alpha_1$ bounds an open disk $E_1$ in $[0,\infty)\times [0,1]$. The pair $(E_1,\alpha)$ is a disc pair and  $\alpha_1$  separates $v$ from $[0,\infty)\times \{0,1\}$. 

We now start an inductive process to obtain our disk pair $(E,\beta)$. Our induction hypothesis is $H_n$: The pair $(E_{n-1},\alpha_{n-1})$ is a disk pair such that $\alpha_{n-1}$ contains $e$ as an edge, $\alpha_{n-1}$ separates the vertices of $E_{n-1}$ from $[0,\infty)\times \{0,1\}$, each edge in the boundary of $\alpha_{n-1}$ is either blue or red and each edge of $\alpha_{n-1}$ belongs to a triangle of $T(\triangle)$. 

If $b$ is an edge of $\alpha_{n-1}$ that belongs to a triangle $\triangle'$ of $T(\triangle)$ not in $E_{n-1}$, then say $(a,b,c)$ is a subpath of $\alpha_{n-1}$. 
Note that  $e$ is not an edge of $\triangle'$ (since $e$ only belongs to one red triangle, $\triangle_k'$). Say boundary edges of $\triangle'$  are $(e_1, e_2, b)$ and the vertex of $\triangle'$ opposite $b$ is $w'$.
We consider three cases.

 {\bf Case 1.} The vertex $w'$ of $\triangle'$ is not a vertex of $\alpha_{n-1}$. 
 In this case, let $\alpha_n$ be the edge path obtained from $\alpha_{n-1}$ by replacing $b$ with $e_1$ and $e_2$.  Then $\alpha_n$ bounds a disk $E_n$ that contains $E_{n-1}$ and $\triangle'$. The pair $(E_n,\alpha_n)$ satisfies our induction hypothesis $H_n$.
 
 {\bf Case 2.}
If $a$ (but not $c$) is in  $\{e_1,e_2\}$ (so $w'$ is a vertex of $a$), then say $a=e_2$. In this case replace $(a,b)$  in $\alpha_{n-1}$ by $e_1$ to form $\alpha_n$. Similarly if $c$ (but not $a$) is in $\{e_1,e_2\}$. Again, The path $\alpha_n$ bounds a disk $E_n$ that contains $E_{n-1}$ and $\triangle'$. The pair $(E_n,\alpha_n)$ satisfies our induction hypothesis $H_n$.

Note that $\{a,c\}\ne \{e_1,e_2\}$ since $\alpha_{n-1}$ is embedded. 

{\bf Case 3} Neither $a$ nor $c$ is an edge of $\triangle'$, but $w'$ is a vertex of $\alpha_{n-1}$. Since $\triangle'$ is not a triangle of $E_{n-1}$ we have the configuration of  Figure 5. Write $\alpha_{n-1}$ as $(\tau_1,\tau_2)$ where $\tau_1$ begins as $(b,c)$ and ends at $w'$, and $\tau_2$ begins at $w'$ and ends with the edge $a$. 

Suppose the edge $e$ belongs to $\tau_1$. Each edge of $\alpha_2$ belongs to a triangle of $T(\triangle)$. These triangles along with $\triangle'$ ``surround"  $e$. But this is impossible since the sequence of triangles $\triangle_1',\ldots \triangle_{k-1}'$ does not contain a triangle of $T(\triangle)$. Instead $e$ is an edge of $\tau_2$.    Let $\alpha_n$ be $(\tau_2,e_2)$ and $E_n$ be the cell bounded by $\alpha_n$. Then $(E_n,\alpha_n)$ is a disk pair satisfying the induction hypothesis $H_n$.

\vspace {.5in}
\vbox to 2in{\vspace {-2in} \hspace {.1in}
\hspace{-.5 in}
\includegraphics[scale=1]{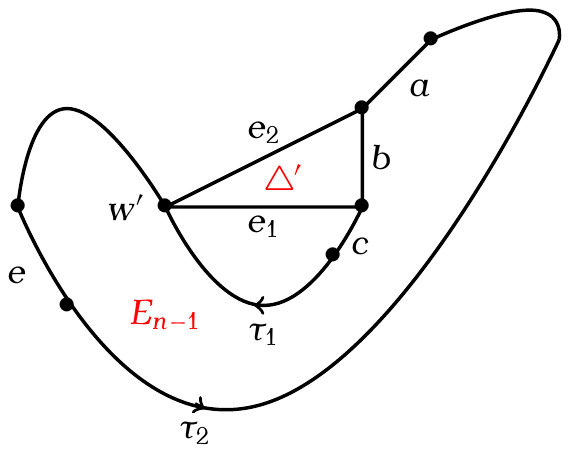}
\vss }

\vspace{-.4 in}

\centerline{Figure 5}

\medskip

If $T(\triangle)$ is finite, our process must terminate with a disk pair $(E,\alpha)$ and each edge of $\alpha$ is blue. The edge $e$ is an edge $\alpha$ and $\triangle_k'$ is in $E$.   If $\triangle'\in S_j$ for $j\geq 1$ and $\triangle'$ is in $E$ then every triangle of $T_j(\triangle')$ is in $E$ (since otherwise the edge path $\alpha $ separates $\triangle'$ and some triangle of $T_j(\triangle')$, but that prevents their equivalence since the edges of $\alpha$ are never red in any coloring). Simply let $\beta=\alpha$. Combining with {\bf (3)}, 
statements $(i)-(iv)$ of the lemma follow.

Next, suppose $T(\triangle)=\{\triangle=\triangle_1, \triangle_2, \ldots\}$ is infinite. This is a more complicated situation than the finite one. 
We proceed as in our earlier induction, but with a bit more care in selecting red edges on $\alpha_i$ for $i>1$. Say that $b$ is the first edge of $\alpha_1$ following $e$ such that $b$ is an edge of a triangle $\hat \triangle$ of $T(\triangle)$ that is not in $E_1$. Such an edge exists, since otherwise, the boundary of $\alpha_1$ has only blue edges and (since $T(\triangle)$ is infinite) some triangle of $T(\triangle)$ is separated from $\triangle_{k}'$ by this blue loop, contrary to the fact that $\triangle_{k}'$ and $\hat \triangle$ are equivalent.  We form $(E_2,\alpha_2)$ using $\hat \triangle$ as in the induction process. To form $(E_3,\alpha_3)$, we let $b$ be the first edge PRECEDING $e$ on $\alpha_2$ that belongs to a triangle of $T(\triangle)$ not in $E_2$, alternating between selecting edges following and preceding $e$ on $\alpha_i$ we produce $(E_n,\alpha_n)$. Again, since $T(\triangle)$ is infinite, this process cannot stop. 
For $n\geq 1$, let $e^n_0=e$ and for $i\geq 1$, let $e^n_i$ be the $i^{th}$ edge of $\alpha_n$ following $e$. Let $e^n_{-i}$ be the $i^{th}$ edge of $\alpha_n$ preceding $e$. 

\medskip
\noindent {\bf Claim} {\it For a fixed integer $i\geq 0$, the $i^{th}$ edge of the $\alpha_n$ eventually stabilize. That is, for  $i\geq 0$ there is a non-decreasing sequence of  integers $K(i)$ such that for all $m\geq K(i)$, $e^m_i=e^{K(i)}_i$ (the $i^{th}$ edge of $\alpha_m$  is equal to the $i^{th}$ edge of $\alpha_{K(i)}$).}

\begin{proof} 
Certainly the Claim is true for $i=0$ with $K(0)=1$. Assume true for an integer $i\geq 0$. 

It is important to observe (because of the stabilization of edges $e_0,\ldots, e_i$) that when $E_{m+1}$ is formed from $E_m$ ($m>K(i)$) and the triangle $\hat\triangle$, Case 3 cannot occur with a vertex of $\hat \triangle$ being the same as a vertex of one of the edges $e_1,\ldots, e_{i-1}$. 

\vspace {.5in}
\vbox to 2in{\vspace {-2in} \hspace {.1in}
\hspace{-.5 in}
\includegraphics[scale=1]{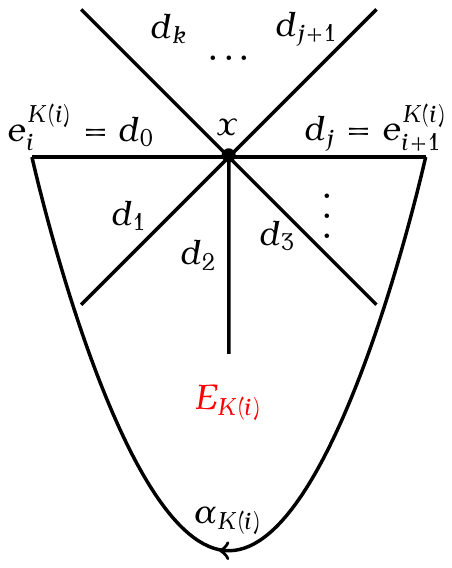}
\vss }

\vspace{.1 in}

\centerline{Figure 6}

\medskip

Let $x$ be the end point of the edge $e^{K(i)}_i$. Let $e_i^{K(i)}= d_0$ and $d_0,d_1,\ldots , d_k$ be the consecutive edges containing $x$ where $d_1$ is in $E_{K(i)}$ (see Figure 6).  Say the edge of $\alpha_{K(i)}$  following $d_0$ is $d_j=e_{i+1}^{K(i)}$ (where $j\in \{1,2,\ldots, k\}$). Then the edges $d_1,\ldots, d_{j-1}$ are in $E_{K(i)}$, so for $m\geq K(i)$, $e^m_{i+1}$ is not in $\{d_0,d_1,\ldots d_{j-1}\}$. Even if $d_j$ is not an edge of a triangle of $T(\triangle)$ that is not in $E_{K(i)}$, it may be that the $i+1^{st}$ edges of the $\alpha_m$ do not stabilize to $d_j$. It may be that for some $p>j$ the edge $d_p$ becomes the $i+1$ edge of $\alpha_m$ for a large value of $m$ (via a move described in Case 3).
But then for $q\geq m$ none of the edges of $\{d_1,\ldots, d_{p-1}\}$ can be the $i+1^{st}$ edge of $\alpha_q$. That means that eventually the $i+1^{st}$ edges of the $\alpha_q$ must stabilize (to $d_s$ for some $s\geq j$) and the Claim is proved.
\end{proof}

The analogue of the Claim with $i<0$ is proved the same way. Let $e_i$ be the stabilized $i^{th}$ edge the $\alpha_n$. Let $\beta$ be the edge path line $(\ldots, e_{-1}, e_0, e_i,\ldots )$ and $E$ be the open half plane on the same side of $\beta$ as $\triangle_{k}'$. 
Part {\it (i)} of Lemma \ref{partition} is true since it is true for each $e_i$. Part {\it (ii)} is true since otherwise $\beta$ would separate $\triangle'$ from another triangle of $T_j(\triangle')$. This is impossible since no edge of $\beta$ is red in any coloring.  Part {\it (iii)} is not applicable here and Part {\it (iv)} follows from the fact that the green vertex $g$ of $\triangle_1'$ belongs to $[0,\infty)\times \{0,1\}$ and each triangle of $\triangle_1',\ldots, \triangle_{k-1}$ is on the side of $\beta$ opposite $E$. This finishes the proof of Lemma \ref{partition}
\end{proof}

All of our Theorem \ref{excise} disk pairs will come from applying Lemma \ref{partition}.  
\vspace {.5in}
\vbox to 2in{\vspace {-1.5in} \hspace {-1.3in}
\hspace{-.5 in}
\includegraphics[scale=1]{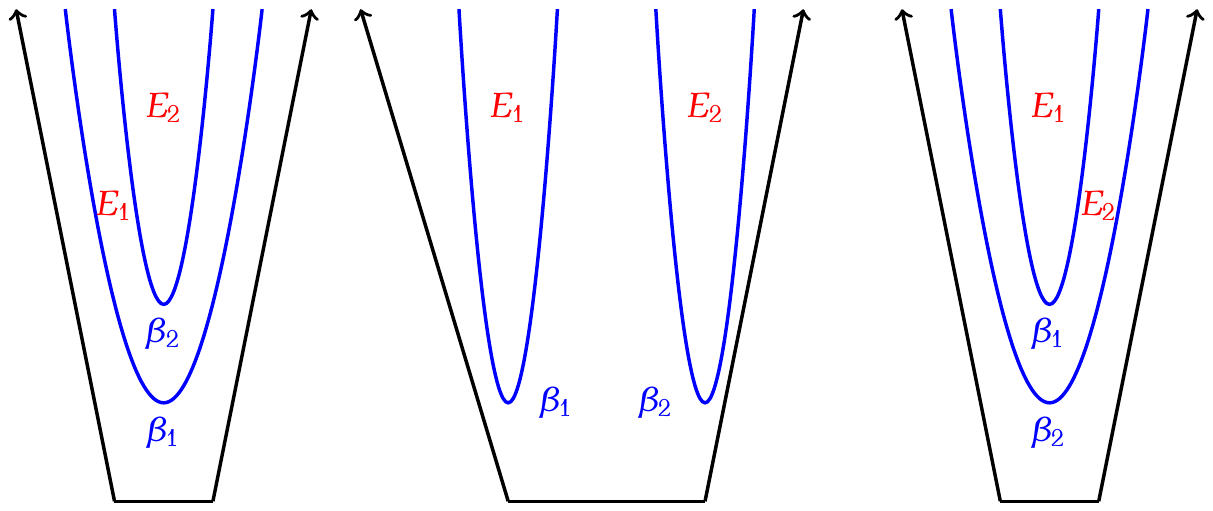}
\vss }

\vspace{-.3 in}

\centerline{Figure 7}

\medskip

\begin{lemma}\label{disjoint1} 
Suppose  $\triangle_1\in S_{i(1)}$, $\triangle _2\in S_{i(2)}$ (where $i(1)$ may equal $i(2)$), and $(E_1,\beta_1)$ and $(E_2,\beta_2)$ are the disk pairs  of Lemma \ref{partition} for $\triangle_1$ and $\triangle_2$ respectively. Then $\beta_1$ does not cross $\beta_2$. If $T(\triangle_1)$ and $T(\triangle_2)$ are distinct equivalence class of triangles of $S_{i(1)}$ and $S_{i(2)}$ respectively, then either $E_1\subset E_2$ or $E_2\subset E_1$ or $E_1\cap E_2=\emptyset$. 
\end{lemma}
\begin{proof} 
If $\beta_1$ and $\beta_2$ cross, then there are edges $e$ and $d$ of $ \beta_1$ such that $e$ is inside $E_2$ and $d$ is outside $E_2$. But then the triangle of $T(\triangle_1)$ containing $e$ is inside $E_2$ and the triangle of $T(\triangle_1)$ containing $d$ is outside $E_2$. This contradicts Lemma \ref{partition} {\it(ii)}, so that $\beta_1$ and $\beta_2$ do not cross. 

Recall that $E_i$ is on the side of $\beta_i$ opposite to $[0,\infty)\times \{0,1\}$. If $\beta_2$ is in the closed half space (or closed disk) determined by $E_1$ and $\beta_1$, then $E_2\subset E_1$ (see Figure 7). If $\beta_1$ is in the closed half space (or closed disk) determined by $E_2$ and $\beta_2$,  then $E_1\subset E_2$. If $\beta_2$ avoids $E_1$ and $\beta_1$ avoids $E_2$, then $E_2\cap E_1=\emptyset$
\end{proof}

For each equivalence class of triangles of the $S_i$ select a disk pair as in Lemma \ref{partition}. List these as $(E_1,\beta_1), (E_2,\beta_2),\ldots$. Let $\mathcal E=\{E_1,E_2,\ldots\}$. Only finitely many of the $\beta_i$ can intersect $[0,k]\times [0,1]$ for any integer $k$, and hence only finitely many of the $E_i$ can intersect $[0,k]\times [0,1]$. Let $K$ be the first integer such some $E_i$ intersects $[0,K]\times [0,1]$, Reorder the $E_i$ so that $E_1,\ldots, E_{i(0)}$ intersect $[0,K]\times [0,1]$ non-trivially and for $j>0$, $E_{i(j-1)+1}, \ldots, E_{i(j)}$ intersects $[0,K+j]\times [0,1]$ non-trivially, but none intersect $[0,K+j-1]\times [0,1]$ non-trivially. 

Observe that if $k<i(j)<m$ then $E_k$ is not a subset of $E_m$. If $j,k\in \{0,1,\ldots, i(0)\}$ and $E_k\subset E_j$ then remove $E_k$ from $\mathcal E$ and reindex (keeping the same order). At this point,  if $1\leq j\leq i(0)$ then $E_j$ is not a subset of $E_m$ for any $m$.  If $j,k\in \{0,1,\ldots, i(1)\}$ and $E_k\subset E_j$ then remove $E_k$ from $\mathcal E$ and reindex (keeping the same order). Continuing, the $E_k$ that remain in $\mathcal E$ satisfy the conclusion of Theorem \ref{excise}.
\end{proof}

The Simplicial Approximation Theorem (see [Theorem 3.4.9, \cite{Span66}]) applies to finite simplicial complexes.  We want a proper version.
\begin{lemma} \label{simpA} 
Suppose $M:[0,\infty)\times [0,1]\to X$ is a proper map to a simplicial complex $X$ where $M_0$ and $M_1$ are (proper) edge paths and $M(0,t)=M(0,0)$ for all $t\in [0,1]$. Then there is a proper simplicial approximation $M'$  of $M$ that agrees with $M$ on $([0,\infty)\times [0,1] )\cup (\{0\}\times [0,1])$.
\end{lemma} 

\begin{proof} 
This result follows from an elementary application of the Simplicial Approximation Theorem.  For $i\in \{1,2,\ldots \}$, let $\alpha_i$ be a simplicial approximation (edge path)  to the path $\beta_i= M|_{ \{i\}\times [0,1]}$. Now $\alpha_i$ is homotopic $rel\{0,1\}$ to $\beta_1$ by a homotopy $H_i$ with image ``close" to the image of $\beta_i$. Alter $M$ to the proper map $M_1$ of Figure 8. 

\vspace {.5in}
\vbox to 2in{\vspace {-2in} \hspace {-1.3in}
\hspace{-.5 in}
\includegraphics[scale=1]{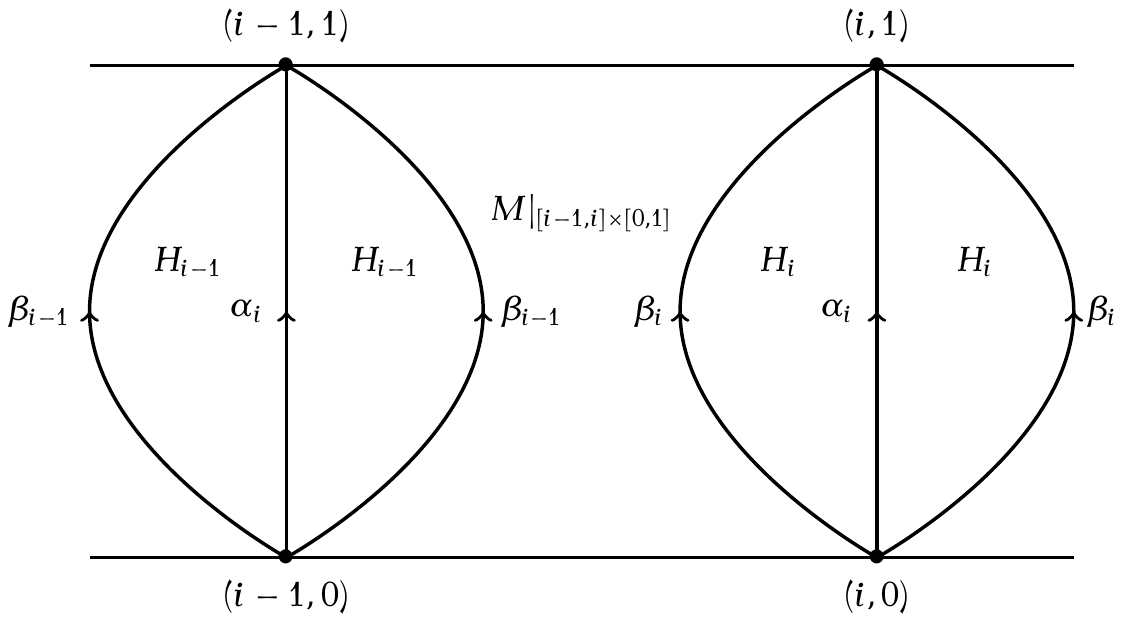}
\vss }

\vspace{.3 in}

\centerline{Figure 8}

\medskip

Note that for $i\in \{1,2,\ldots \}$, $M_1$ is simplicial on the boundary of $[i-1,i]\times [0,1]$. Simply apply the Simplicial Approximation Theorem to $M_1$ restricted to $[i-1,i]\times [0,1]$ for each $i$. Since the approximation agrees with $M$ on the boundary of $[i-1,i]\times [0,1]$, these approximations can be patched together to produce $M'$.
\end{proof}

\section{The Proof of Theorem \ref{Hom}}\label{T1} 

Let $M$ be a proper homotopy with image in $X$ (relative to a base point $\ast$)  from the proper $Y$-edge path ray $r$ to the proper $Y$-edge path ray $s$. So $M:[0,\infty)\times [0,1]\to X$ is a proper map such that $M(t,0)=r(t)\in Y$, $M(t,1)=s(t)\in Y$ and $M(0,t)=\ast$ for all $t$. We apply Theorem \ref{excise} with matching notation, so that $X$ is our cusped space for $(G,\mathcal P)$, $Y\subset X$ is the Cayley 2-complex of $G$ and $Z_i$ are the Cayley 2-complexes for the cosets of the peripheral subgroups $P\in\mathcal P$. 

Let $(E_j,\alpha_j)_{j\in J}$ be the disk pairs in $[0,\infty)\times [0,1]$ given by Theorem \ref{excise}. Formally,  $\alpha_j$ is an edge path mapping of  a circle or real line to $\partial E_j$ (the boundary of $E_j$), $M$ maps $\partial E_j$ to $Z_{i(j)}$ and 

$$M([0,\infty)\times [0,1]-\cup_{j\in J}E_j)\subset Y$$

Using the fact that the peripheral  subgroups of $\mathcal P$ have semistable first homology at $\infty$, we will  attach 2-manifolds to the (circle/lines) $\partial E_j$ and map these 2-manifolds into $Y$ (such that on  $\partial E_j$, these maps agree with $M$) in a ``uniformly proper" way. In this way, we obtain a 2-manifold manifold $N$ (with boundary the line $([0,\infty)\times \{0,1\})\cup (\{0\}\times [0,1])$) and a proper map of $N$ into $Y$ that agrees with $M$ on $([0,\infty)\times [0,1])-\cup_{j\in J}E_j$. This will show that $r$ and $s$  are properly homologous in $Y$, as desired.

Suppose $C$ is a compact subset of $Y$. If $P\in \mathcal P$ and $\Gamma$ is a copy of the Cayley 2-complex of $P$ in $Y$, then only finitely many $G$-translates of $\Gamma$ intersect $C$. Hence there are only finitely many Cayley 2-complexes $\Gamma_1,\ldots ,\Gamma_m$ such that $\Gamma_i$ is a copy of a Cayley 2-complex in $Y$ for some $P\in \mathcal P$ and $C\cap \Gamma_i\ne \emptyset$. Each $P\in \mathcal P$ (and hence each $\Gamma_i$) has semistable first homology at $\infty$. By Theorem \ref{H1SS} there is $D_i(C)$ compact in $\Gamma_i$ such that if  $\alpha$ is a proper edge path line or edge path loop in  $\Gamma_i-D_i$, there is a 2-manifold $N_\alpha$ with single boundary component a line/circle $S_\alpha$, and proper map $L_\alpha:N_\alpha\to \Gamma_i-C$ such that $L_\alpha$ restricted to $S_\alpha$ is $\alpha$. Let $D(C)= D_1(C)\cup \cdots \cup D_m(C)$

Choose compact sets $C_0\subset C_1\subset \cdots$ in $Y$ so that $C_i\cup D(C_i)$ is a subset of the interior of $C_{i+1}$ for $i\geq 1$ and $\cup_{i=1}^\infty C_i=Y$. 
Let $C_0=\emptyset$. Recall $\alpha_j$ is an edge path mapping of a circle or line to  $\partial E_j\subset [0,\infty)\times [0,1]$.
Let $N(j)$ be the largest integer such that $M(\alpha_j)$ has image in $Y-C_{N_j}$. There is a 2-manifold $N_j$ with single boundary component $S_j$ and a proper map $L_j:N_j\to Y-C_{N_j-1}$ 
so that $L_j$ on $S_j$ agrees with $M$ on  $\partial E_j$. (By this we mean, there is a homeomorphism  $h_j:S_j\to \partial E_j$ such that $Mh_j=L_j$ on $S_j$). 

Attach $N_j$ to $\partial E_j$ by the attaching homeomorphism $h_j$ and extend $M$ to $N_j$ by $L_j$. After attaching  $N_j$ for all $j\in J$, the result is a 2-manifold $N$ and we have a map $\hat M:N\to Y$ that agrees with $M$ on $[0,\infty)\times [0,1]-\cup_{j\in J}E_j$. All that is left to show is that $\hat M$ is proper. 

Let $C$ be compact in $Y$. Pick $n$ such that $C\subset C_n$. Since $\hat M$ and $M$ agree on $[0,\infty)\times [0,1] -\cup_{j\in J}E_j$, the set   $\hat M|_{[0,\infty)\times [0,1] -\cup_{j\in J}E_j}^{-1} (C)$ is compact. For any $j\in J$, if $M(\partial E_j)\subset Y-C_{n+1}$ then $\hat M(N_j)\cap C_n=\emptyset$. Since the $E_j$ are mutually disjoint and each is a union of open triangles, there  are only finitely many $j$ such that $\partial E_j \cap C_{n+1}$ is non-empty. Then for all but finitely many $j$, $\hat M|_{N_j}^{-1}(C)=\emptyset$. Since $\hat M$ is proper on each $N_j$, $\hat M^{-1}(C)$ is a finite union of compact sets and hence $\hat M$ is proper.  
This completes the proof of Theorem \ref{Hom}.
\bibliographystyle{amsalpha}
\bibliography{paper1}{}

\newcommand{\etalchar}[1]{$^{#1}$}
\def\cprime{$'$}
\providecommand{\bysame}{\leavevmode\hbox to3em{\hrulefill}\thinspace}
\providecommand{\MR}{\relax\ifhmode\unskip\space\fi MR }
\providecommand{\MRhref}[2]{%
  \href{http://www.ams.org/mathscinet-getitem?mr=#1}{#2}
}
\providecommand{\href}[2]{#2}
\begin{thebibliography}{ABC{\etalchar{+}}91}

\bibitem[ABC{\etalchar{+}}91]{ABC91}
J.~M. Alonso, T.~Brady, D.~Cooper, V.~Ferlini, M.~Lustig, M.~Mihalik,
  M.~Shapiro, and H.~Short, \emph{Notes on word hyperbolic groups}, Group
  theory from a geometrical viewpoint ({T}rieste, 1990), World Sci. Publ.,
  River Edge, NJ, 1991, Edited by Short, pp.~3--63. \MR{1170363}

\bibitem[BM91]{BM91}
Mladen Bestvina and Geoffrey Mess, \emph{The boundary of negatively curved
  groups}, J. Amer. Math. Soc. \textbf{4} (1991), no.~3, 469--481. \MR{1096169}

\bibitem[BM17]{BM17}
Indranil Biswas and Mahan Mj, \emph{{$H_1$}-semistability for projective
  groups}, Math. Proc. Cambridge Philos. Soc. \textbf{162} (2017), no.~1,
  89--100. \MR{3581900}

\bibitem[Bow99]{Bow99B}
B.~H. Bowditch, \emph{Connectedness properties of limit sets}, Trans. Amer.
  Math. Soc. \textbf{351} (1999), no.~9, 3673--3686. \MR{1624089}

\bibitem[Bow01]{Bow01}
\bysame, \emph{Peripheral splittings of groups}, Trans. Amer. Math. Soc.
  \textbf{353} (2001), no.~10, 4057--4082. \MR{1837220}

\bibitem[Bow04]{Bo04}
Brian~H. Bowditch, \emph{Planar groups and the {S}eifert conjecture}, J. Reine
  Angew. Math. \textbf{576} (2004), 11--62. \MR{2099199}

\bibitem[CM14]{CM2}
Gregory~R. Conner and Michael~L. Mihalik, \emph{Commensurated subgroups,
  semistability and simple connectivity at infinity}, Algebr. Geom. Topol.
  \textbf{14} (2014), no.~6, 3509--3532. \MR{3302969}

\bibitem[DS05]{DS05}
Cornelia Dru\c{t}u and Mark Sapir, \emph{Tree-graded spaces and asymptotic
  cones of groups}, Topology \textbf{44} (2005), no.~5, 959--1058, With an
  appendix by Denis Osin and Mark Sapir. \MR{2153979}

\bibitem[Dun85]{Dun85}
M.~J. Dunwoody, \emph{The accessibility of finitely presented groups}, Invent.
  Math. \textbf{81} (1985), no.~3, 449--457. \MR{807066}

\bibitem[Far74]{FFT}
F.~Thomas Farrell, \emph{The second cohomology group of {$G$} with
  {$Z\sb{2}G$}\ coefficients}, Topology \textbf{13} (1974), 313--326.
  \MR{0360864}

\bibitem[Geo08]{G}
Ross Geoghegan, \emph{Topological methods in group theory}, Graduate Texts in
  Mathematics, vol. 243, Springer, New York, 2008. \MR{2365352}

\bibitem[GM85]{GM85}
Ross Geoghegan and Michael~L. Mihalik, \emph{Free abelian cohomology of groups
  and ends of universal covers}, J. Pure Appl. Algebra \textbf{36} (1985),
  no.~2, 123--137. \MR{787167}

\bibitem[GM08]{GMa08}
Daniel Groves and Jason~Fox Manning, \emph{Dehn filling in relatively
  hyperbolic groups}, Israel J. Math. \textbf{168} (2008), 317--429.
  \MR{2448064}

\bibitem[Mih87]{M87}
Michael~L. Mihalik, \emph{Semistability at {$\infty$}, {$\infty$}-ended groups
  and group cohomology}, Trans. Amer. Math. Soc. \textbf{303} (1987), no.~2,
  479--485. \MR{902779}

\bibitem[MS]{MS18}
Michael~L. Mihalik and Eric Swenson, \emph{Relatively hyperbolic groups with
  semistable fundamental group at infinity}, ArXiv: 1709.02420 [math.GR].

\bibitem[MW]{MW18}
Jason~F. Manning and Oliver Wang, \emph{Cohomology and the {B}owditch
  boundary}, arXiv:1806.07074 [math.GR].

\bibitem[Osi06]{Osin06}
Denis~V. Osin, \emph{Relatively hyperbolic groups: intrinsic geometry,
  algebraic properties, and algorithmic problems}, Mem. Amer. Math. Soc.
  \textbf{179} (2006), no.~843, vi+100. \MR{2182268}

\bibitem[Spa66]{Span66}
Edwin~H. Spanier, \emph{Algebraic topology}, McGraw-Hill Book Co., New
  York-Toronto, Ont.-London, 1966. \MR{0210112}

\bibitem[Swa96]{Swarup}
G.~A. Swarup, \emph{On the cut point conjecture}, Electron. Res. Announc. Amer.
  Math. Soc. \textbf{2} (1996), no.~2, 98--100. \MR{1412948}

\end{thebibliography}

\end{document}